\documentclass[smallcondensed]{svjour3}     
\smartqed  
              
\usepackage{graphicx}
\usepackage[english]{babel}
\usepackage{subcaption}
\usepackage{graphics}
\usepackage{graphicx}
\usepackage{epsfig}
\usepackage{amssymb}
\usepackage{arydshln}
\usepackage{amsmath}
\usepackage{amsfonts}
\usepackage{float}
\usepackage{multirow} 
\usepackage{booktabs}
\usepackage{physics}
\usepackage{pstricks,pst-plot}
\usepackage{hyperref}
\usepackage{epstopdf}
\usepackage{dsfont}
\usepackage{tikz}
\usetikzlibrary{decorations.pathreplacing,angles,quotes}
\usetikzlibrary{decorations.pathmorphing,patterns}

\newcommand{\be}{\begin{equation}}
\newcommand{\ee}{\end{equation}}
\newcommand{\ba}{\begin{array}}
\newcommand{\ea}{\end{array}}

\newcommand{\comment}[1]{}

\newcommand{\B}[1]{\mbox{\boldmath $#1$}}

\journalname{Adv. Comp. Math.}

\begin{document}

\title{Computing the Reciprocal of a $\phi$-function by Rational 
Approximation\thanks{The first and third author acknowledge support by the INdAM-GNCS 2019
 project {\em Analisi di matrici sparse e data-sparse: metodi numerici ed 
applicazioni}.}}

\author{Paola Boito         \and
        Yuli Eidelman  \and
        Luca Gemignani
}

\authorrunning{P. Boito, Y. Eidelman and L. Gemignani} 

\institute{P. Boito \at
Dipartimento di Matematica, Universit\`{a} di Pisa,
Largo Bruno Pontecorvo, 5 - 56127 Pisa, Italy\\
\email{paola.boito@unipi.it}
\and
Y. Eidelman \at
School of Mathematical Sciences, Raymond and Beverly
Sackler Faculty of Exact Sciences, Tel-Aviv University, Ramat-Aviv,
69978, Israel\\
\email{eideyu@tauex.tau.ac.il}
\and
L. Gemignani \at
Dipartimento di Informatica, Universit\`{a} di Pisa,
Largo Bruno Pontecorvo, 3 - 56127 Pisa, Italy\\
\email{luca.gemignani@unipi.it}
}

\date{Received: date / Accepted: date}



\maketitle

\begin{abstract}
In this paper we  introduce a family of rational  approximations of 
the reciprocal of a $\phi$-function  involved in  the explicit solutions of  
certain linear differential equations,  as well as in integration schemes 
evolving on manifolds. The derivation and properties of this family of approximations applied to scalar and matrix arguments are presented.
Moreover, we 
show that the matrix functions computed by these approximations exhibit decaying properties  comparable to  the 
best  existing theoretical bounds. Numerical examples highlight the benefits of the proposed rational  approximations  w.r.t.~the classical  Taylor polynomials and other rational functions. 
\end{abstract}

\section{Introduction}\label{sec:intro}
Numerical methods for the  computation of matrix functions have witnessed
growing interest in recent years (see \cite{HighamBook}, \cite{HighamCatalogue}
and the references given therein). 
One important class of applications is the solution of some classical problems for  ordinary or partial 
differential equations. Several methods have been developed for the 
evaluation of the matrix $\phi$-functions $\phi_k(A)$, $k\geq 0$, where $A$ is a large and possibly sparse matrix, and $\phi_k(z)$ are entire functions defined recursively by $\phi_{k+1}(z)=\displaystyle\frac{\phi_k(z)-(1/k!)}{z}$
with $\phi_0(z)=e^z$ \cite{HoOs}.

Here we  focus on  the related  issue of  approximating  the matrix function 
$\psi_1(A)$  where $\psi_1(z)$ is a meromorphic
function defined as the reciprocal of $\phi_1(z)$, that is,
\[
\psi_1(z)=\phi_1(z)^{-1}=\frac{z}{e^z-1}
\]
and $A$ is banded or more generally rank-structured (see \cite{EGH_book}  for a
survey  on such matrices).   This  problem  also  plays an important  role
in a number of applications. We describe two of these  applications in more 
detail.

\subsection{Applications}
Two-point inverse problems for first 
order differential equations are frequently encountered in mathematical 
physics (see Chapter 7 in \cite{POV_book}). As a model, in this paper
we consider the  differential  problem 
\begin{equation}\label{bvp}
  \displaystyle\dv{\B u}{t}=A \B u(t) +\B p, \quad  0\leq t\leq \tau,
  \end{equation}
where $A\in \mathbb R^{d\times d}$ is given, while 
$\B p\in  \mathbb R^d$ is unknown.   In order to find the solution
$\B u\colon [0, \tau]\rightarrow \mathbb R^d$  of \eqref{bvp}
 and the vector $\B p$  simultaneously,
the overdetermined conditions
\begin{equation}\label{bvpc}
  \B u(0)=\B u_0=\B g, \quad \B u(\tau)=\B h,
\end{equation}
can be imposed.  Note that a more general formulation of the inverse problem
\eqref{bvp}, \eqref{bvpc} in a Banach space
with a closed linear operator $A$ is treated in \cite{TE1,T1}, whereas 
a new formula for the solution of the problem (\ref{bvp}),
(\ref{bvpc}) using Bernoulli polynomials is given in \cite{ETS}.

Now assume that the complex numbers
\begin{equation}\label{cnm}
2\pi ik/\tau,\quad k=\pm1,\pm2,\dots
\end{equation}
do not belong to the spectrum of $A$. Define the complex-valued functions
\begin{equation}\label{qw}
q_t(z)=\frac{z}{e^{\tau z}-1}e^{zt},\; 
w_t(z)=\frac{e^{zt}-1}{e^{\tau z}-1},\quad 0\le t\le\tau,\;z\in\mathbb{C}
\end{equation}
with $q_t(0)=\frac1{\tau},\;w_t(0)=\frac{t}{\tau}$. The complex 
functions $q_t(z),w_t(z)$ are meromorphic in $z$ with poles (\ref{cnm}).
One can check directly that the solution of the inverse problem (\ref{bvp}), 
(\ref{bvpc}) is given by the formulas
\begin{equation}\label{fm1}
\B p=q_0(A)(\B h-\B g)-A\B g
\end{equation}
and
\begin{equation}\label{fm2}
\B u(t)=w_t(A)(\B h-\B g)+\B g,\quad 0\le t\le\tau.
\end{equation}
Using (\ref{fm1}) and the formula $q_0(z)=\psi_1(\tau z)/\tau$ we obtain the formula
\begin{equation}\label{pformula}
\B p= \frac{1}{\tau}\psi_1(\tau A)\left(\B h-\B g\right)-A\B g. 
\end{equation}
to compute the unknown vector $\B p$ via the function $\psi_1$.

Notice also that the formula
$$
\B v(t)=q_t(A)\B v_0,\quad 0\le t\le\tau
$$
yields the solution of the nonlocal problem 
\begin{equation}\label{laris}
  \displaystyle\dv{\B v}{t}=A \B v(t), \;  0\leq t\leq \tau,\quad
\int_0^{\tau}v(t)\;dt=v_0
\end{equation}
studied by the authors in \cite{BEG2018}.

Computing the inverse of $\phi_1(A)$, with $A\in \mathbb R^{d\times d}$, is also a fundamental task in the application of exponential integrators for the 
numerical solution of systems of differential equations. The reason is twofold.
First,  certain  integration schemes called Runge-Kutta Munthe-Kaas (RKMK) methods 
\cite{M1,M2,IN} for computing numerical solutions  of differential equations 
that are guaranteed to evolve  on a prescribed manifold  require explicitly  
the approximation of the function $\psi_1$ applied to a matrix. More precisely, suppose that $G$ is a finite-dimensional Lie group acting transitively on a smooth manifold $\mathcal{M}$. In many classical examples, ${\mathcal M}=G$ is a matrix Lie group, acting on itself by left or right multiplication. Denote by $\mathfrak{g}$ the Lie algebra of $G$ and let $p$ be a fixed base point in $\mathcal{M}$. Any smooth curve $y(t)$ on $\mathcal{M}$ in a neighborhood of $p$ can be seen as the image of a curve $\sigma(t)$ through the origin of $\mathfrak{g}$ via the exponential mapping:
$$y(t)=\exp(\sigma(t))\cdot p,\qquad \sigma(0)=0.$$
A differential equation for $y(t)$ on the manifold takes the form
$\frac{\rm d}{{\rm d}t}y=F(y)$, where $F$ is a vector field on $\mathcal{M}$,
and it can be reformulated as a differential equation for $\sigma(t)$:
$$
\frac{\rm d}{{\rm d}t}\sigma(t)={\rm d}\exp_{\sigma}^{-1}(f(\exp(\sigma)\cdot p)),
$$
where $f:\mathcal{M}\longrightarrow\mathfrak{g}$ is a suitable representation of the vector field $F$; see e.g. \cite{M2} or \cite{celle} for details.
This equation holds on a linear space, where one may apply a standard Runge-Kutta method. A crucial step in doing so is the evaluation of the reciprocal of the differential of the exponential map:
$$
{\rm d}\exp_{\sigma}^{-1}(v)=\left.\frac{z}{\exp(z)-1}\right|_{z={\rm ad}_{\sigma}}v,
$$
where ${\rm ad}_{\sigma}$ denotes the commutator:
${\rm ad}_{\sigma}(v)=[\sigma,v]$. In other words, in the matrix manifold case each RKMK step requires the computation of $\psi_1([M,N])$, where $M\in \mathbb R^{d\times d}$ is fixed and $[M,N]=MN-NM$ is the matrix commutator.  
Including the evaluation of the map $\psi_1([M,N])$  
in numerical algorithms  seems to be awkward and several polynomial 
approximations of $\psi_1(z)$ have been presented in the  related literature 
(compare with \cite{celle} and the references given therein).

As a second remark on the role of approximating $\psi_1(A)$ in the context of exponential integrators, we observe that the 
study of reliable  procedures for the evaluation of  $\psi_1(A)$  and   
$\psi_1(A)\B v$  based on rational approximations of the meromorphic function
$\psi_1(z)$  might  be used to foster the development of rational Krylov 
methods for  computing  $\phi_1(A)$  and   $\phi_1(A)\B v$, which is the main 
computational bulk in exponential integrators for stiff systems of differential
equations \cite{GG}.   Indeed the  properties of these methods depend heavily 
on the features of the underlying rational approximations  for the selection of
the  poles and of the subspace of approximation. 

\subsection{Approximation of $\psi_1(A)$}
Customary approximations of $\psi_1(A)$  derived from truncated  Taylor  
expansions   go back to the work of
 Magnus \cite{Mag}.
 These approximations are quite accurate if the norm of the matrix $A$ is sufficiently small.  On the other hand, rational functions
 may exhibit approximation properties  and convergence domains superior to polynomials
provided that the poles of the rational functions
involved have been chosen in a suitable way.
Moreover, if $A$ is banded or even just rank-structured
then the same property  holds   in a certain  approximate sense for the matrix
$\phi_1(A)$ and thus a fortiori for its  inverse $\psi_1(A)$. Polynomial approximations for the  function  $\psi_1(z)$
often require a quite high degree of the approximating polynomial in order to achieve
a reasonable quality of approximation of the numerical rank structure and the decaying properties of the matrix $\psi_1(A)$.
Rational
approximations would  typically obtain the
same quality with substantially fewer degrees of freedom.

In this paper we present new algorithms that efficiently approximate the  functions  of a matrix argument
involved in the solution of  (\ref{bvp}),(\ref{bvpc}).   In particular,
we propose a  novel family of fixed-poles mixed polynomial-rational approximations of $\psi_1(z)$  required for the
computation of the vector $\B p$  according to
\eqref{pformula}.
By combining
Fourier analysis methods applied to the function $q_t(z)/z$
with classical tools for   Fourier series acceleration  \cite{ECK}  for any fixed  $s>1$ and  $m\geq s$
we obtain  approximations of $\psi_1(A)=(\phi_1(A))^{-1}$ of the form
\begin{equation}\label{appformula}
\psi_1(A)\simeq p_s(A)+\sum_{k=1}^m \gamma_{k,s}A^{\tau_s}(A^2 +k^2 I_d)^{-1},
\end{equation}
where $p_s(z)$ is a polynomial of degree $\ell=\ell(s)$ and $\tau_s=\tau(s)\in \mathbb N$.   These  novel  expansions
compares favorably with polynomial approximations based on the Maclaurin series as well as  other rational Pad\'e  approximants determined by
inverting the approximation of $\phi_1(A)$. 
Specifically: 
\begin{enumerate}
 \item  Theory and  numerical evidence show that
the formulas \eqref{appformula}  are accurate  on
larger domains  than their polynomial counterparts  thus allowing for larger  steps in integration schemes.
\item  Typically rational approximations  based on  \eqref{appformula} and Pad\'e   techniques  behave similarly.  However,
  the computation $\psi_1(A)$ and $\psi_1(A)\B v$ by means of \eqref{appformula} is insensitive of  numerical difficulties due to the conditioning
  of $\phi_1(A)$. Also, computations  based on \eqref{appformula} are inherently parallelizable. 
\item  Besides this, the scheme  \eqref{appformula} can be applied easily and efficiently to
  remarkable classes of matrices including band,
  rank structured and displacement structured matrices, which are often found in applications
  (e.g., from discretization of differential operators). Indeed, fast and robust inversion algorithms are available for these classes of matrices, together with cheap storage techniques.
  In particular,  when $A$ is rank-structured
  the action of the matrix $\psi_1(A)$ on a  vector can be computed efficiently using the  direct fast solver for
  shifted linear systems proposed in \cite{BEG2018}.  
  \item For  a symmetric banded matrix $A$
these novel approximations \eqref{appformula}
  yield a  computable reconstruction of the  associated matrix function  $\psi_1(A)$ which exhibits
  decaying properties  comparable to the best  existing theoretical bounds and significantly superior to
  the behavior of the  corresponding polynomial approximations. The matrix  function $\psi_1(A)$ can  thus be
  manipulated efficiently using  its resulting  data-sparse format  combined
  with the rank-structured matrix technology \cite{EGH_book}.
  \item Note that the rational part of \eqref{appformula} has poles $\{\pm ik\}_{k=1,\ldots,m}$. The choice of a fixed set of poles can be advantageous in view of application to rational Krylov methods for computing $\psi_1(A)$ or $\psi_1(A)\B v$, as well as for error analysis. 
\end{enumerate}

\subsection{Structure of the paper}
  The paper is organized as follows. In Section \ref{twos} we
  present a general  scheme for the design of accurate rational approximations
  of the matrix functions involved in the solution of
  (\ref{bvp}),(\ref{bvpc}). In  Subsection \ref{sub1}   this scheme is specialized  for the construction
  of mixed polynomial-rational approximations  of
  the meromorphic function $\psi_1(z)$. An application to a multi-degree of freedom physical system is illustrated in Section \ref{sec:multidof}. In Section \ref{three} we investigate  both theoretical and computational properties of
  the application of mixed polynomial-rational  formulas   to computing $\psi_1(A)$ where  $A$ is a symmetric banded matrix. Finally,
  conclusions and future work are presented in Section \ref{four}.

\section{Rational Approximation of the Inverse  Problem and the Reciprocal of 
the $\phi_1$-function}\label{twos}


The solvability of the  inverse problem \eqref{bvp}, \eqref{bvpc} in an 
abstract Banach space is studied in \cite{TE1,T1,ETS}. Under the assumption 
that all the numbers (\ref{cnm}) are regular 
points of the  linear operator $A$,   the  inverse problem \eqref{bvp},
\eqref{bvpc} has a unique solution.  Without loss of generality one can assume 
that $\tau=2\pi$. As it was mentioned above in the matrix case this solution is
given by the formulas (\ref{qw}), (\ref{fm1}), (\ref{fm2}). 

We introduce the auxiliary function   
\begin{equation}\label{r}
r_t(z)=\frac{e^{zt}}{e^{z2\pi}-1},\quad 0\le t\le 2\pi.
\end{equation}
Using the formulas \eqref{qw} we have
\begin{equation}\label{rqw}
q_t(z)=zr_t(z),\;w_t(z)=r_t(z)-r_0(z),\quad 0\le t\le 2\pi.
\end{equation}
 Expanding the function $r_t(z)$ in the Fourier series of $t$ we obtain 
\[
r_t(z)=\frac1{2\pi}\left(\frac{1}{z}+\sum_{k\in{\mathbb Z}/\{0\}}
\frac{e^{ikt}}{z-ik}\right),\quad 0<t<2\pi.
\]
We consider the equivalent representation with the real series  given by 
\begin{equation}\label{realexp}
r_t(z)=\frac1{2\pi z}+\frac1{\pi}\sum_{k=1}^{\infty}(z\cos(k t)-k\sin (k t))
(z^2+k^2)^{-1},\quad 0<t<2\pi.
\end{equation}
It is well known that the 
convergence of this series  must depend strongly on the smoothness of the  
periodic extension of $r_t(z)$. Acceleration techniques proposed in \cite{ECK} 
make use of the Bernoulli polynomials for the approximate reconstruction of 
jumps.  

Applying the formula 
\begin{equation}\label{fm5}
\frac{1}{z^2+k^2}=\frac1{k^2}-\frac{z^2}{k^2(z^2+k^2)}
\end{equation}
to the last entry in (\ref{realexp}) we obtain that for $0<t<2\pi$ it holds
\[
r_t(z)=\frac1{2\pi z}+\frac1{\pi}
\sum_{k=1}^{\infty}(z\cos(k t)+\frac{1}{k}z^2\sin (k t)))(z^2+k^2)^{-1}-
\frac1{\pi}\sum_{k=1}^{\infty}\frac{1}{k}\sin(k t). 
\]
Since
\[
2\sum_{k=1}^{\infty}\frac1{k}\sin kt=\pi-t, \quad 0<t<2\pi
\]
we  arrive at the following formula for $0<t<2\pi$, 
\begin{equation}\label{el2}
r_t(z)=\frac1{2\pi z}+\frac{t-\pi}{2\pi}+y_t(z)
\end{equation}
with
\begin{equation}\label{mash}
y_t(z)=\frac1{\pi}\sum_{k=1}^{\infty}(z\cos(k t)+\frac{1}{k}z^2\sin(k t))
(z^2+k^2)^{-1}.
\end{equation}
 If $z\in K\subset \mathbb C$, $K$ compact set,  then  definitively we have  
\begin{equation}\label{ir1}
\left| z-\mathrm i k \right|^{-1}\le\frac{C}{|k|},
\end{equation}
and, therefore, using the Weierstrass M-test one can  easily check
that the series in (\ref{el2}) converges uniformly in 
$(t, z)\in[0,2\pi]\times K$. Hence,  by  continuity we  may extend the formula
\eqref{mash} over the whole interval  $[0,2\pi]$.

Combining the formulas (\ref{qw}) and (\ref{el2}) we get
\begin{equation}\label{mash1}
q_t(z)=\frac1{2\pi}+z\frac{t-\pi}{2\pi}+zy_t(z)
\end{equation}
and
\begin{equation}\label{mash2}
w_t(z)=\frac{t}{2\pi}+(y_t(z)-y_0(z)).
\end{equation}
Here there are no singularities at $z=0$. Inserting (\ref{mash1}), 
(\ref{mash2}) in (\ref{fm1}), (\ref{fm2}) we obtain the formulas for the 
solution of the inverse problems
\begin{equation}\label{fm3}
\B p=(\frac1{2\pi}I-\frac1{2}A+Ay_0(A))(\B h-\B g)-A\B g
\end{equation}
and
\begin{equation}\label{fm4}
\B u(t)=\left(\frac{t}{2\pi}I+(y_t(A)-y_0(A))\right)
(\B h-\B g)+\B g,\quad 0\le t\le2\pi
\end{equation}
with 
\begin{equation}\label{mash3}
y_t(A)=\frac1{\pi}\sum_{k=1}^{\infty}(A\cos(k t)+\frac{1}{k}A^2\sin(k t))
(A^2+k^2I)^{-1}.
\end{equation}

The rate of convergence of the series in \eqref{mash3} is the same as for the 
series $\sum_{k=1}^{\infty}k^{-2}$. It can be improved by using 
repeatedly the equality (\ref{fm5}) as above. For each  integer $k\geq 0$ 
denote as $B_k(t)$ the Bernoulli polynomials (extended by periodicity  onto the
real line) defined by
 \begin{equation}\label{poly}
 \frac{z e^{z t}}{e^z-1}=\sum_{k=0}^{+\infty} B_k(t) \frac{z^k}{k!}, \quad 
|t|<2 \pi, 
 \end{equation}
 \[
 B_k=B_k(0), \quad k\geq 0,
 \]
where $B_k$ are the Bernoulli numbers. Then,  using (\ref{fm5}) 
we prove by induction the following formulas. 
\begin{lemma}
  Let $t\in [0, 2\pi]$ and $y_t(z)\colon \mathbb C\rightarrow \mathbb C$  be defined as in
    \eqref{mash}.  Then  we have 
    \begin{equation}\label{nov27}
y_t(z)=p_{n,t}(z)+s_{n,t}(z),\quad n=0,1,2,\dots
\end{equation}
with 
\begin{equation}\label{novv27}
p_{n,t}(z)=\sum_{i=2}^{2n+1}\frac{(2\pi)^{i-1}}{i!}B_i
\left(\frac{t}{2\pi}\right)z^{i-1}
\end{equation}
and 
\begin{equation}\label{novvv27}
s_{n,t}(z)=\frac{(-1)^n}{\pi}
\sum_{k=1}^{\infty}\frac{z^{2n}(z\cos(k t)+\frac{1}{k}z^2\sin (k t)))}
{k^{2n}(z^2+k^2)}.
\end{equation}
\end{lemma}
\begin{proof}
For $n=0$ the relation (\ref{nov27}) follows directly from 
(\ref{mash}). Assume that for some $n\ge0$ the relation (\ref{nov27}) holds.
Using (\ref{fm5}) we have
\[
\begin{array}{c}
s_{n,t}(z)=\frac{(-1)^{n+1}}{\pi}
\displaystyle
\sum_{k=1}^{\infty}\frac{z^{2(n+1)}(z\cos(k t)+\frac{1}{k}z^2\sin (k t)))}
{k^{2(n+1)}(z^2+k^2)}+\\
\frac{(-1)^n}{\pi}\displaystyle
\sum_{k=1}^{\infty}\frac{z^{2n}(z\cos(k t)+\frac{1}{k}z^2\sin (k t)))}
{k^{2(n+1)}},
\end{array}
\]
i.e.
\begin{equation}\label{masha}
s_{n,t}(z)=s_{n+1,t}(z)+b_{n,t}(z),
\end{equation}
where 
\[
b_{n,t}(z)=\frac{(-1)^n}{\pi}z^{2n+1}
\sum_{k=1}^{\infty}\frac{\cos(k t)}{k^{2(n+1)}}+
\frac{(-1)^n}{\pi}z^{2n+2}
\sum_{k=1}^{\infty}\frac{\sin(k t)}{k^{2n+3}}.
\]
Using the formulas from [\cite{AS}, formula 23.1.18] we know that
\[
\sum_{k=1}^{\infty}\frac{\cos(k t)}{k^{2(n+1)}}=
\frac{(2\pi)^{2(n+1)} B_{2(n+1)}(t/(2\pi))}{(-1)^n 2 (2 (n+1))!}
\]
and
\[ 
\sum_{k=1}^{\infty}\frac{\sin(k t)}{k^{2n+3}}=
\frac{(2\pi)^{2n+3} B_{2(n+1)}(t/(2\pi)}{(-1)^n 2 (2n+3)!}.
\]

Hence it follows that 
\[
b_{n,t}(z)=\frac{z^{2n+1}(2\pi)^{2n+1} B_{2n+2}(t/(2\pi))}
{(2n+2)!}+\frac{z^{2n+2}(2\pi)^{2n+2} B_{2n+3)}(t/(2\pi))}
{(2n+3)!}.
\]
Inserting this in (\ref{masha}) and using (\ref{nov27})
we complete the proof of the statement.
\end{proof}



\subsection{The application to the  $\psi_1$-function}\label{sub1}

Observe that $q_0(z)=\psi_1(2 \pi z)/(2\pi)$ and
$\psi_1(z)$ admits a Maclaurin series expansion  which can virtually be used to evaluate
$q_0(A)$. The following classical result provides the  Maclaurin expansion of
$\psi_1(z)$.
\begin{theorem}[\cite{AS}, formula 23.1.1] It holds
  \[
   \psi_1(z)={\phi_1(z)}^{-1}=\sum_{k=0}^{+\infty} \frac{B_k}{k!}z^k, \quad |z|< 2 \pi, 
    \]
    where $B_k$ denotes  the $k$th  Bernoulli number. 
  \end{theorem}

Different  rational approximations of 
$\psi_1(z)$ can  be  derived  from the Fourier  series expansion of
 $r_t(z)$. It turns out that  such a  series representation  is also related with the
 Mittag-Leffler  expansion  of $q_0(z)$.

Specifying the formulas and representations obtained above to the function
$\psi_1(A)$ we obtain the following. Using the formula (\ref{mash1}) we have
\begin{equation}\label{nats}
q_t(A)=\frac1{2\pi}I+\frac{t-\pi}{2\pi}A+Ay_t(A)
\end{equation}
and using (\ref{mash3}) we find that
\begin{equation}\label{fone}
  \begin{array}{ll}
q_t(A)=\frac1{2\pi}I_d+\frac{t-\pi}{2\pi}A+\frac1{\pi}
\displaystyle\sum_{k=1}^{\infty}A^2\cos(k t)(A^2+ k^2I_d)^{-1}+\\
 +\frac1{\pi}\displaystyle
\sum_{k=1}^{\infty}\frac{1}{k}A^3\sin( k t)(A^2+k^2I_d)^{-1}, 
\quad 0\leq t\leq \tau=2\pi, 
  \end{array}
\end{equation}
which implies
\begin{equation}\label{fone1}
  \begin{array}{ll}
\psi_1(A)=2\pi q_0(A/(2\pi))=I_d-\frac{1}{2}A +
2\displaystyle\sum_{k=1}^\infty\left(\frac{A}{2 \pi}\right)^{2}
    (\left(\frac{A}{2 \pi}\right)^2+ k^2I_d)^{-1}. 
  \end{array}
\end{equation}
Relation \eqref{fone1}  is the first member of our family of rational 
approximations of $\psi_1(A)$. 

This result may be improved by applying repeatedly the same approach as above. Indeed
using (\ref{nats}) and (\ref{nov27}), (\ref{novv27}), (\ref{novvv27})
\[
\begin{array}{ll}
q_t(A)=\frac1{2\pi}I+\frac{t-\pi}{2\pi}A+\displaystyle
\sum_{i=2}^{2n+1}\frac{(2\pi)^{i-1}}{i!}B_i\left(\frac{t}{2\pi}\right)A^i+\\
\frac{(-1)^n}{\pi}\displaystyle
\sum_{k=1}^{\infty}\frac{A^{2n+1}(A\cos(k t)+\frac{1}{k}A^2\sin (k t)))}
{k^{2n}}(A^2+k^2I)^{-1}.
\end{array}
\]
Setting $t=0$ we get
\[
q_0(A)=\frac1{2\pi}I-\frac1{2}A+\displaystyle
\sum_{i=2}^{2n+1}\frac{(2\pi)^{i-1}}{i!}B_iA^i+
\frac{(-1)^n}{\pi}\displaystyle
\sum_{k=1}^{\infty}\frac{A^{2n+2}}{k^{2n}}(A^2+k^2I)^{-1}.
\]
Since for $n>1$ the odd Bernoulli numbers $B_n$ are zeroes we have
$$
\displaystyle\sum_{i=2}^{2n+1}\frac{(2\pi)^{i-1}}{i!}B_iA^i=
\displaystyle\sum_{i=0}^{n-1}
\frac{(2\pi)^{2i+1}}{(2(i+1))!}B_{(2(i+1))}A^{2(i+1)}.
$$

 Hence, using $\psi_1(A)={\phi_1( A)}^{-1}=2\pi q_0(A/(2\pi))$  we arrive at 
the  main  result of the present paper
\begin{theorem}\label{main}
   For any fixed $n>0$ it holds
\[
 \psi_1( A) =p_n(A) +
 2(-1)^{n}\displaystyle\sum_{k=1}^{\infty}\left(\frac{A}{2 \pi}\right)^{2(n+1)}
    \frac{1}{k^{2n}}(\left(\frac{A}{2 \pi}\right)^2+ k^2I_d)^{-1},
    \]
    where 
 \[
  p_n(A)=  I_d -\frac{1}{2} A+ 
\displaystyle\sum_{i=0}^{n-1} A^{2(i+1)} \frac{B_{2(i+1)}}{(2 (i+1))!}.
 \]
\end{theorem}   

 Observe that $p_n(A)$ is the  classical approximation of  $\psi_1(A)$ given  
in  Theorem 1.  Also notice that the rate of convergence of the series is the 
same as for the series $\sum_{k=1}^{\infty}k^{-2(n+1)}$ where $2n$ is the 
degree of the polynomial approximation. The above result presents a rational 
correction of this  approximation  aimed to improve its  convergence 
properties. Specifically, based on Theorem \ref{main} we introduce the 
following family $\{\psi_{n, s}(A)\}_{(n,s)\in \mathbb N\times \mathbb N}$
 of mixed polynomial-rational approximations of $\psi_1(A)$:
 \begin{equation}\label{eq:approx}
 \psi_{n, s}(A)=p_n(A) + 2 (-1)^{n}\left(\displaystyle\sum_{k=1}^{s}
 \frac{1}{k^{2n}}\left(\left(\frac{A}{2 \pi}\right)^2+ k^2I_d\right)^{-1}\right)
\left(\frac{A}{2 \pi}\right)^{2(n+1)}.
 \end{equation}
 
 \begin{remark}
   The  above approach based on the Fourier series expansion of $q_t(z)/z$  encompasses  some 
   rational approximations of $\psi_1(z)$  which can 
   also be derived by applying Mittag-Leffler pole decomposition
   (see e.g., \cite{arfken} for a concise, hands-on presentation) to the function $q_0(z)$. More precisely, 
 let us apply formula (7.54) in  \cite{arfken} to $q_0(z)=\psi(2\pi z)=\frac{2\pi z}{e^{2\pi z-1}}$ with $p=1$.
 The poles of our function are
 $\{ik\}_{k\in\mathbb{Z}\setminus\{0\}}$ and the corresponding
 residues are readily seen to be $\{ik\}_{k\in\mathbb{Z}\setminus\{0\}}$ as well.
So we have
\begin{eqnarray*}
&&q_0(z)=q_0(0)+zq'(0)+\sum_{k\in\mathbb{Z}\setminus\{0\}}\frac{ikz^2/(ik)^2}{z-ik}=\\
&&=1-\pi z+ \sum_{k\in\mathbb{Z}\setminus\{0\}}\frac{-iz^2}{k(z-ik)}= \\
&&=1-\pi z + \sum_{k=1}^{\infty}\frac{2z^2}{z^2+k^2},\\
\end{eqnarray*}
which is exactly formula \eqref{el2} with $t=0$. 
At this point we can apply \eqref{fm5}  and proceed as above (again with $t=0$) to obtain: 
\begin{eqnarray}
&&q_0(z)=1-\pi z+2\sum_{k=1}^{\infty}z^2\left(\sum_{i=0}^{n-1}(-1)^i\frac{z^{2i}}{k^{2i+2}}+(-1)^n\frac{z^{2n}}{k^{2n}(z^2+k^2)}\right)=\nonumber \\
&&=1-\pi z+2\sum_{i=0}^{n-1}(-1)^iz^{2i+2}\sum_{k=1}^{\infty}\frac{1}{k^{2i+2}}+2(-1)^n z^{2n}\sum_{k=1}^{\infty}\frac{1}{k^{2n}(z^2+k^2)}\nonumber \\
&&=1-\pi z+2\sum_{i=0}^{n-1}(-1)^iz^{2i+2}\zeta(2i+2)+2(-1)^n z^{2n}\sum_{k=1}^{\infty}\frac{1}{k^{2n}(z^2+k^2)},\label{series}
\end{eqnarray}
where $\zeta$ denotes the Riemann zeta function.
  Now recall that even-indexed Bernoulli numbers are characterized by the relation
\begin{equation}
B_{2\ell}=\frac{(-1)^{\ell-1}(2\ell)!}{2^{2\ell-1}\pi^{2\ell}}\zeta(2\ell)\label{bernoulli}
\end{equation}
(see e.g. \cite{gradshteyn}, item 9.616), whereas the odd-indexed ones are zero except for $B_1=-\frac12$. From \eqref{bernoulli} we deduce
\begin{equation}
\zeta(2\ell)=\frac{B_{2\ell}(-1)^{\ell-1}2^{2\ell-1}\pi^{2\ell}}{(2\ell)!}\label{zeta}
\end{equation}
and by plugging \eqref{zeta} with $\ell=i+1$ in equation \eqref{series} we obtain 
$$
q_0(z)=1-\pi z+\sum_{i=0}^{n-1}\frac{B_{2i+2}(2\pi z)^{2i+2}}{(2i+2)!}+2(-1)^n z^{2n}\sum_{k=1}^{\infty}\frac{1}{k^{2n}(z^2+k^2)},
$$
which is essentially the same mixed polynomial-rational development as in Theorem \ref{main}, in scalar form.
 \end{remark}
 

\begin{figure}
\begin{center}
\includegraphics[width=0.9\textwidth]{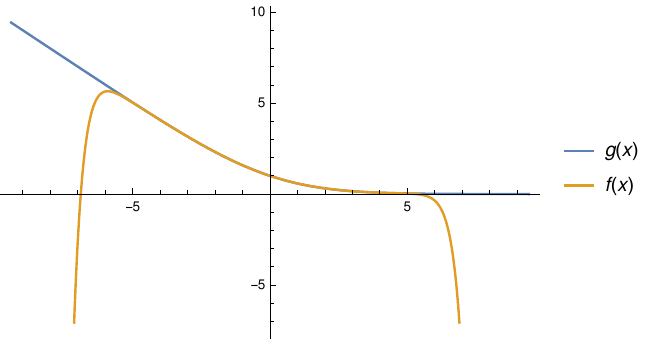}
\caption{Polynomial approximation $f(x)=\psi_{20,0}(x)$ against the function $g(x)=\psi_1(x)$.}\label{fig1}
\end{center}
\end{figure}
\begin{figure}
\begin{center}
\includegraphics[width=0.9\textwidth]{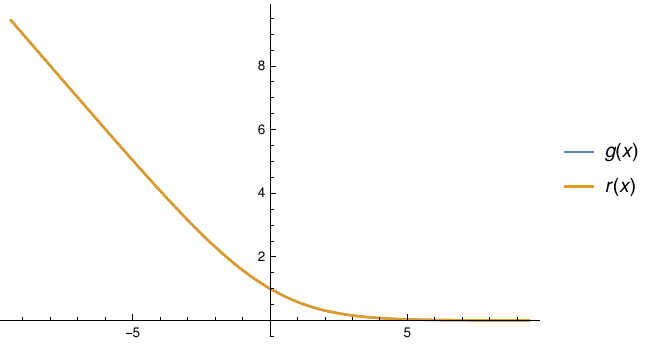}
\caption{Rational approximation $r(x)=\psi_{4,16}(x)$ against the function $g(x)=\psi_1(x)$.  The two plots overlap.}\label{fig2}
\end{center}
\end{figure}

\subsection{Numerical experiments}
We begin by testing the behavior of mixed approximations applied to scalar (real or complex) arguments. 
 
In  Figures \ref{fig1} and \ref{fig2}  we  show the plot over the interval $ [-3 \pi,   3 \pi]$
 of the functions $g(x)=\psi_1(x)=\displaystyle\frac{x}{e^x -1}$, 
 its  polynomial approximation  $f(x)=\psi_{20,0}(x)$ and its  rational approximation $r(x)=\psi_{4,16}(x)$. Clearly,
 the  rational approximation performs better when the points are close to the border of the convergence disk of
 the Maclaurin series  given in Theorem 1. This same phenomenon  can be observed in the complex plane. 
 In Figure \ref{fig4} we  illustrate the absolute error of rational approximation
 at complex points $x=a +\mathrm i b$ with $a, b \in [-3 \pi,   3 \pi]$. 

\begin{figure}
\begin{center}
\includegraphics[width=0.6\textwidth]{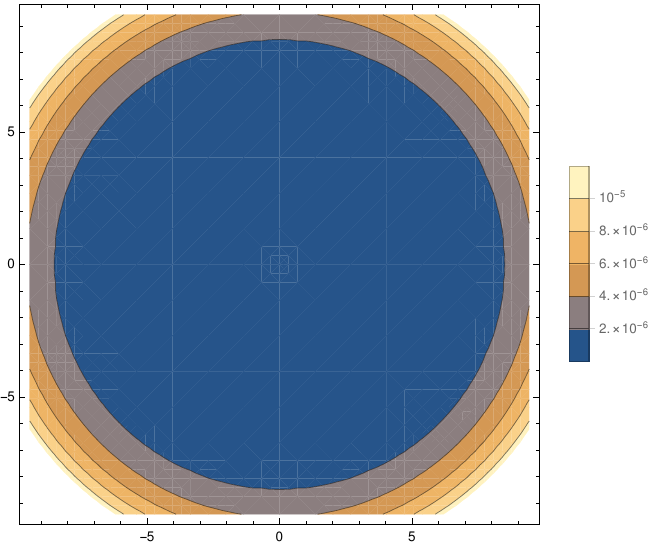}
\caption{Absolute error of rational  approximation. }\label{fig4}
\end{center}
\end{figure}

\begin{remark}\label{rem:degrees}
It is interesting to compare the complexity of computing a mixed approximation $\psi_{n,s}(A)$ and a classical rational approximation to $\phi_1(A)$, such as diagonal $(k,k)-$Pad\'e, when $A$ is a large structured matrix. 

Suppose for instance that $A$ has size $d\times d$ with quasiseparable rank $h$, in which case structured inversion of $A$ requires $\mathcal{O}(h^2 d)$ operations. Recall that a polynomial of degree $k$ applied to $A$ yields a structured matrix of quasiseparable rank $hk$. Then the cost of applying a $(k,k)-$Pad\'e approximation $N(z)/D(z)$ to $A$ is dominated by the computation of $N(A)D(A)^{-1}$, which requires $\mathcal{O}(h^2 k^2 d)$ operations. On the other hand, the evaluation of $\psi_{n,s}(A)$, where $n$ is supposed to be small and constant, is dominated by the computation of the $s$ rational terms, whose cost amounts to $\mathcal{O}(s h^2 d)$ operations. In other words, the computational cost tends to grow quadratically with the degree of a Pad\'e approximation, whereas it grows linearly with the degree of a mixed approximation. 

\end{remark}

\begin{remark}\label{rem:mixedss}
A widespread approach to the computation of exponential and $\phi_\ell$ functions combines polynomial or Pad\'e approximation with a few steps of scaling-and-squaring \cite{higham2009scaling}. In principle, scaling-and-squaring may also be applied to our mixed polynomial-rational approximation, scaling the function argument by a suitable power of $2$ and then making use of the squaring formulas
\begin{equation}
\psi_1(2z)=\frac{2\psi_1(z)}{e^z+1}=\frac{2\psi_1(z)^2}{z+2\psi_1(z)}.\label{psisquaring}
\end{equation}
See Figure \ref{fig:scalarapprox} for a numerical example.
 
Since mixed approximation is accurate on a larger domain than polynomial or Pad\'e, it requires a smaller number of squaring steps. This is a useful property, because in some cases each squaring step may contribute to significant error accumulation (see e.g., \cite{dieci2000pade,al2009new} and references therein). 

Unfortunately, the application of \eqref{psisquaring} when computing $\psi_1(A)$ or $\psi_1(A)\B v$ requires the inversion of the matrix $A+2\psi_1(A)$ or of $e^A+I$. Note however that the matrix $e^A+I$ will generally be well-conditioned for symmetric $A$, even when $A$ has negative eigenvalues.  
\end{remark}

\begin{figure}

\includegraphics[width=0.45\textwidth]{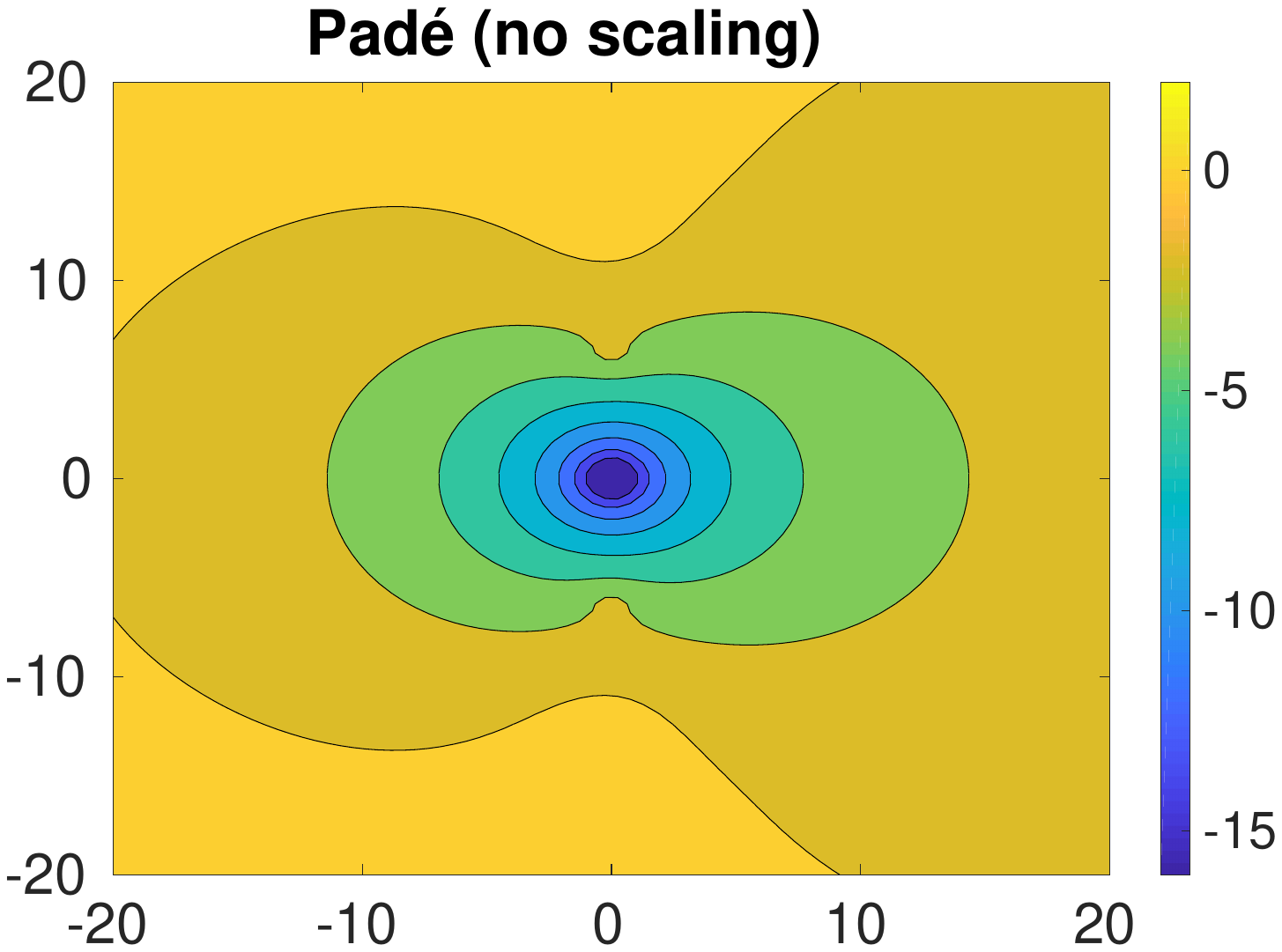} 
\includegraphics[width=0.45\textwidth]{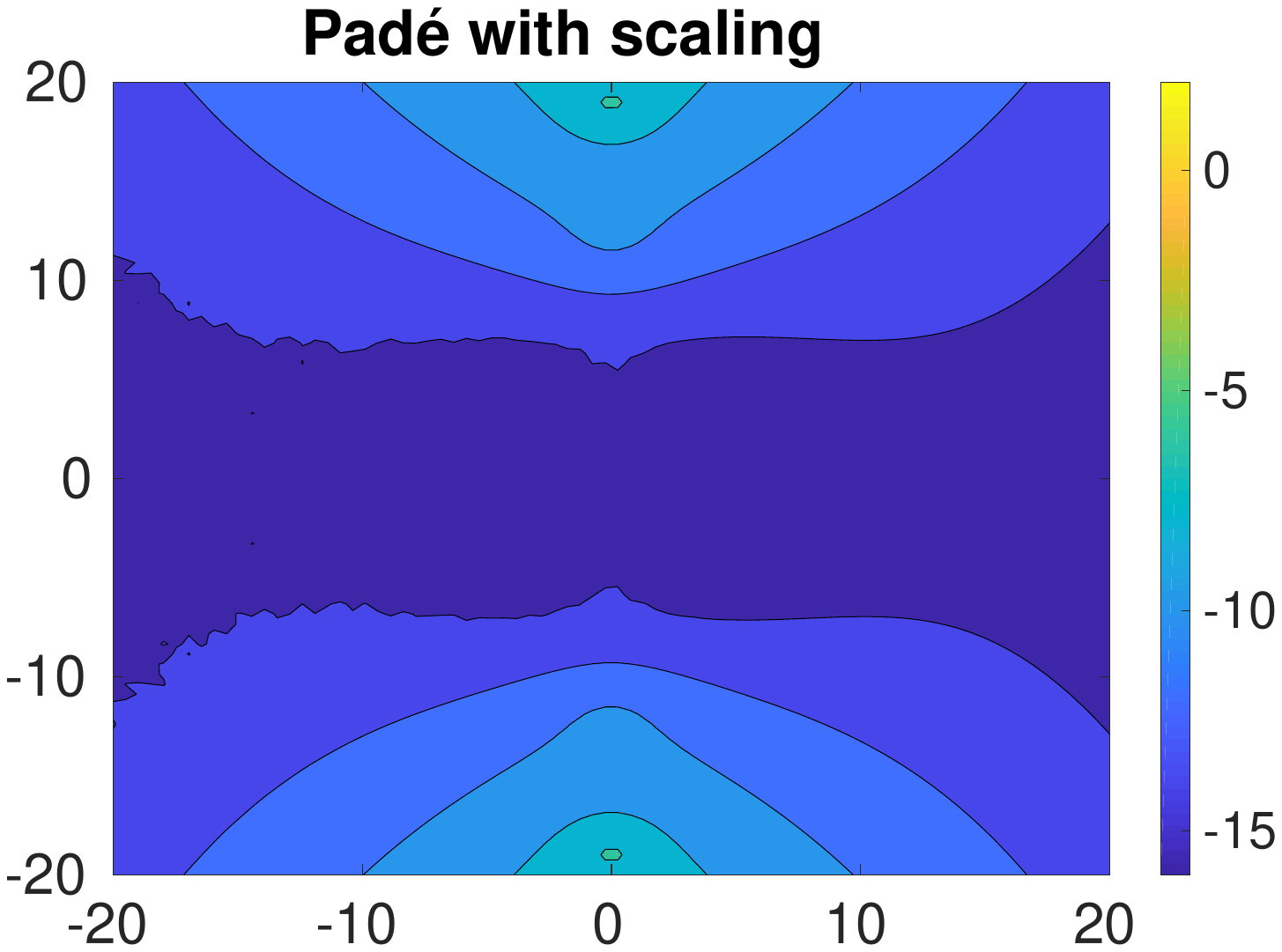} 

\includegraphics[width=0.45\textwidth]{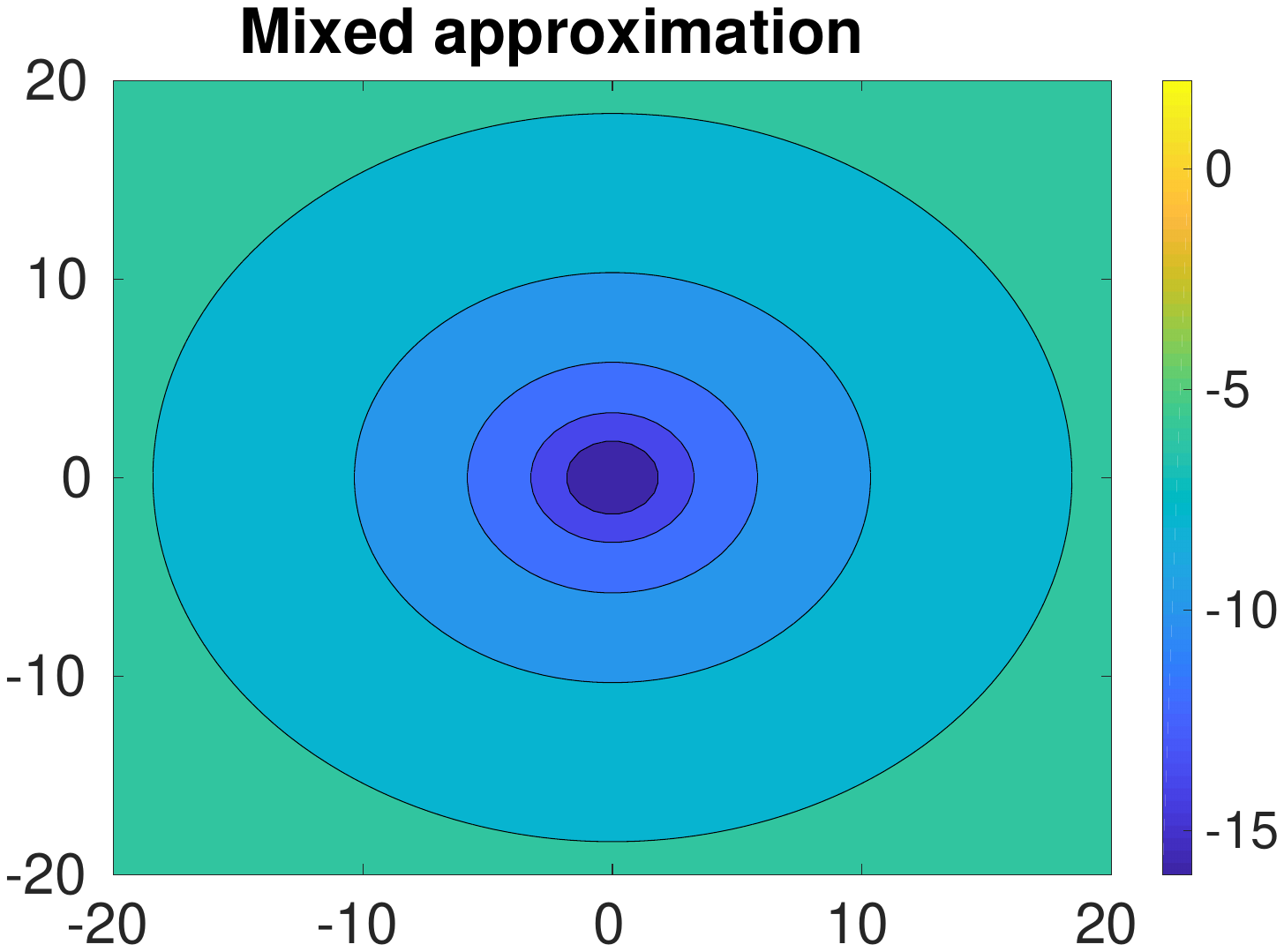} 
\includegraphics[width=0.45\textwidth]{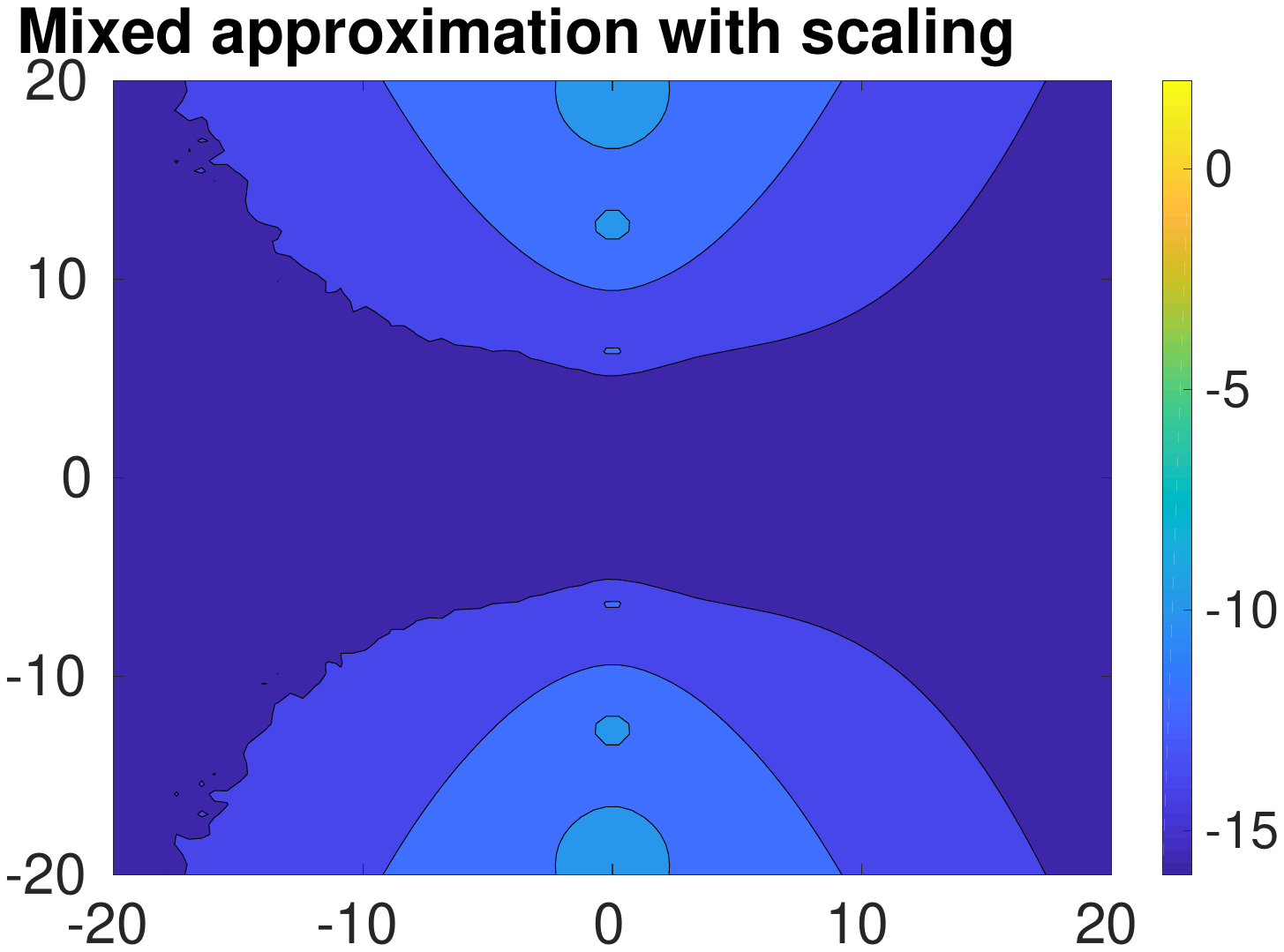} 

\includegraphics[width=0.45\textwidth]{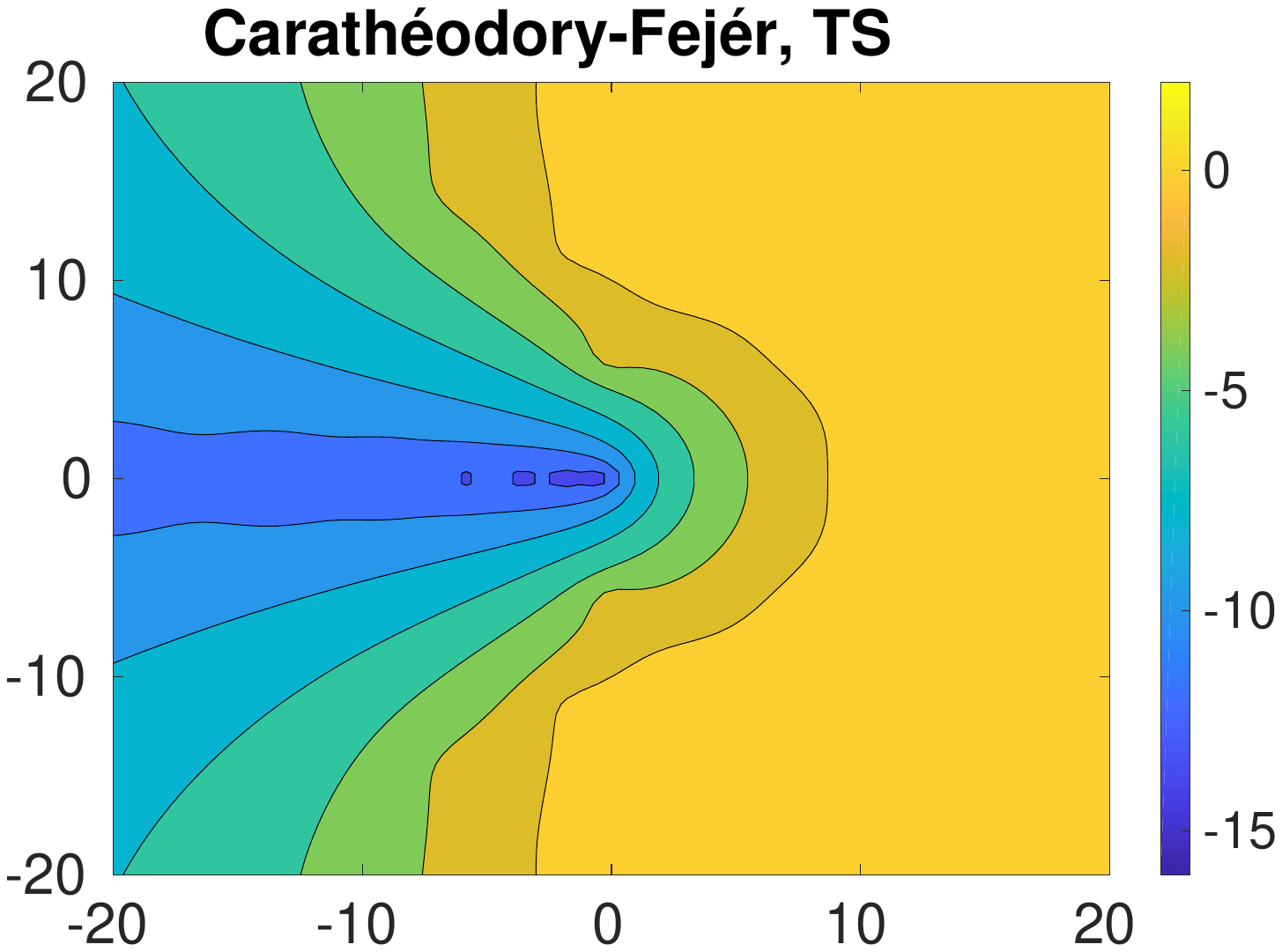} 
\includegraphics[width=0.45\textwidth]{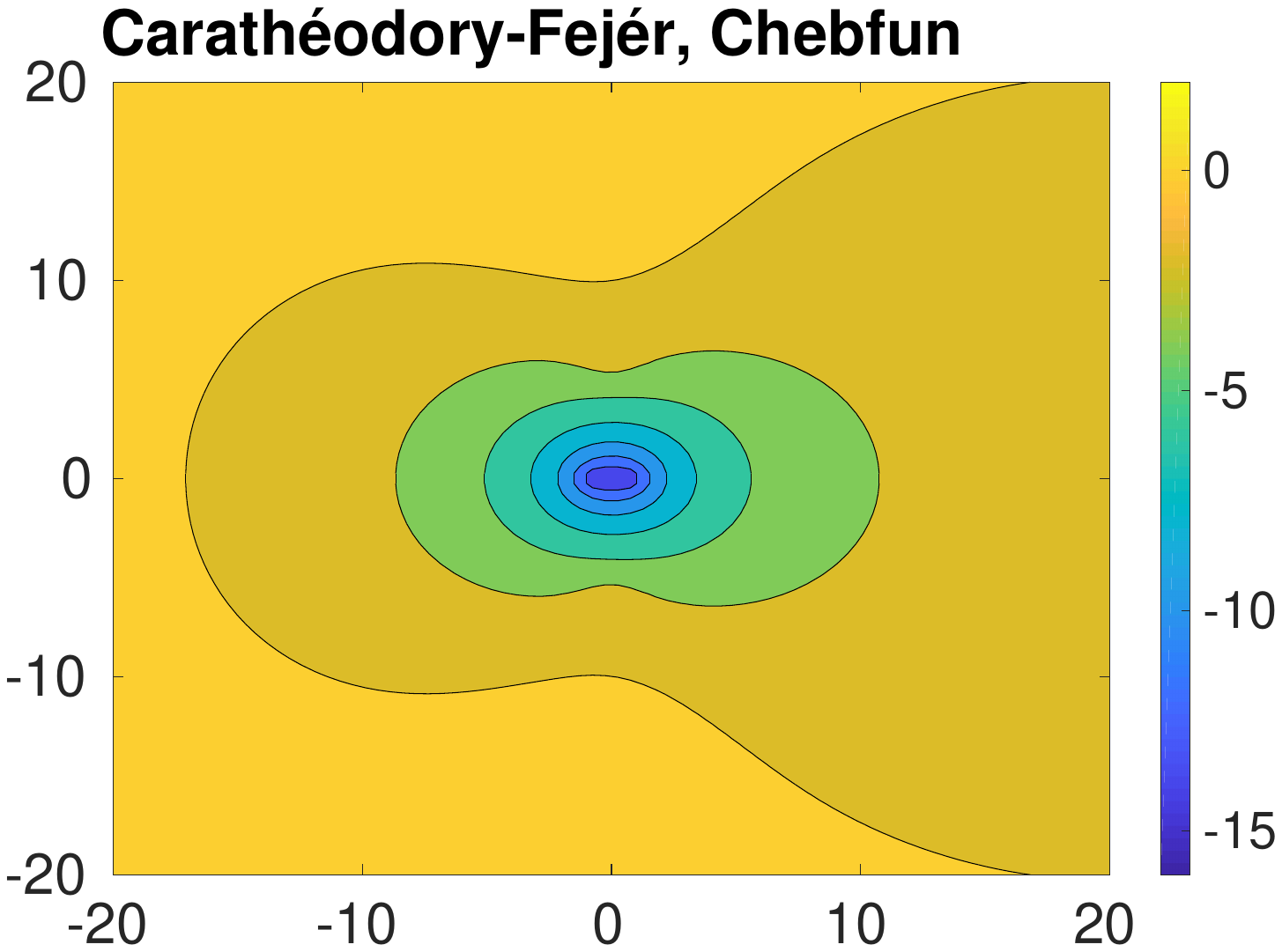} 

\caption{Absolute error for the approximation of the function $\psi_1(z)$ in the complex plane, represented in $\log_{10}$ scale. Here Pad\'e approximations are obtained as the reciprocals of the $(6,6)$-Pad\'e approximations to $\phi_1(z)$ given by {\tt padeapprox} in Chebfun \cite{driscoll2014chebfun} (without scaling) and by {\tt phipade} in EXPINT \cite{berland2007expint} (with scaling). Mixed approximations are computed as $\psi_{3,20}(z)$. The last two figures show results obtained using Carath\'eodory-Fej\'er approximation implemented as in \cite{ST} (left) and as in Chebfun (right), both without scaling.}\label{fig:scalarapprox}
\end{figure}

In the next tables we  compare the accuracy of polynomial and rational approximations for computing
both the matrix function $\psi_1(A)$ and the vector $\psi_1(A)\B b$, where $A\in \mathbb R^{d\times d}$ is symmetric
and
$\B b\in \mathbb R^d$. As claimed in the introduction
we are  interested in the case where $A$ is structured
so that we can assume that a  linear system $A \B x=\B f$ can  be solved in linear time (possibly up to logarithmic factors)
with a linear storage. All of our numerical tests are performed using MATLAB R2019a\footnote{The MATLAB code used for the numerical tests is available at {\tt http://people.cs.dm.unipi.it/boito/psi1.zip}.} For comparison purposes we use the Pad\'e approximation to the $\phi_1$-function as implemented in the EXPINT package \cite{berland2007expint}. 

Our test suite is as follows:

\begin{enumerate}

\item\label{test-poisson} $A$ is the $900\times 900$ block tridiagonal matrix obtained by discretizing the 2-dimensional Laplace operator with the usual 5-point rule on an $30\times 30$ grid, that is,
$A={\tt {gallery('poisson', 30)}}$;

\item\label{test-tridiag} $A$ is the  $d\times d$ Toeplitz tridiagonal  matrix  generated  by the command  $A={\tt {gallery('tridiag', d, -1, 4,-1)}}$, i.e.,
$$
A=\left[\begin{array}{rrrrr}
4 & -1 & 0 & \dots & 0\\
-1 & 4 &-1 & \ddots& \vdots\\
0 & \ddots &\ddots & \ddots & 0\\
\vdots&\ddots&\ddots&\ddots & -1\\
0&\dots & 0 &-1&4
\end{array}\right].
$$

\item\label{test-invtridiag}  $A$ is the order one $d\times d$ quasiseparable matrix generated as
  \[
  A=0.7*{\tt {inv(gallery('tridiag', d, d/2, [d:-1:1], d/2 ))}};
\] that is:
$$
A=\left[\begin{array}{rrrrr}
d & d/2 & 0 & \dots & 0\\
d/2 & d-1 & d/2 & \ddots& \vdots\\
0 & \ddots &\ddots & \ddots & 0\\
\vdots&\ddots&d/2&2 & d/2\\
0&\dots & 0 &d/2&1
\end{array}\right],
$$
with even $d$.

\item\label{test-kms} $A$ is  the $d\times d$ Kac-Murdock-Szeg\"o Toeplitz matrix such that $A_{ij}=0.8^{|i-j|}$. It can be generated in MATLAB as
$A={\tt {gallery('kms',d,0.8)}}$;

\item\label{test-smoke} $A$ is the $100\times 100$ ``smoke matrix'' generated as $A={\tt {gallery('smoke', 100)}}$. This is a nonsymmetric matrix.
\end{enumerate}

For comparison purposes, when $A$ is symmetric an accurate approximation  of $\psi(A)$  is determined by computing the spectral decomposition of $A$. If $A$ is nonsymmetric, an accurate approximation of $\psi(A)$ is determined using MATLAB's VPA environment with 32 digits.

For any given $n\geq 1$ the polynomial  and rational approximations  of  ${\phi_1( A)}^{-1}$ are  $\psi_{n,0}(A)$ and
$\psi_{m,n-m}(A)$, respectively.  The corresponding normwise relative errors are
\[
{err\_p}_n=\frac{\parallel \psi(A) -p_n(A)\parallel}{\parallel \psi(A)\parallel}, \quad
{err\_r}_{n,s}=\frac{\parallel \psi(A) -\psi_{n,s}(A)\parallel}{\parallel \psi(A)\parallel}.
\]
Tables \ref{table:poisson}, \ref{t1}, \ref{t2}, \ref{t3} and \ref{t4}   show the errors evaluated for polynomial, rational Pad\'e and mixed approximations (without scaling) applied to the test matrices. Pad\'e approximations are computed by inverting the approximation of $\phi_1(A)$ given by EXPINT; in these examples $\phi_1(A)$ is generally well conditioned.

At this time we are  just  interested in  comparing the accuracy of
different approximations without 
incorporating fast linear solvers in our code.  However, we point out that  for the considered examples fast
solvers exist that are expected to behave like Gaussian-elimination-based algorithms.   Observe that in
Example (\ref{test-invtridiag}) for $d=2056$ the
eigenvalues are out of the  disk   centered at the origin of radius $2 \pi$ and this explains the divergent behavior of the 
polynomial approximation. 

\begin{table}
\begin{center}
\caption{Errors for Example (\ref{test-poisson}). Polynomial approximation fails in this example (i.e., it yields errors $\gg 1$).}\label{table:poisson}
\begin{tabular}{cc}
\hline\noalign{\smallskip}
$k$& ${err\_}(k,k)$-Pad\'e \\ 
\noalign{\smallskip}\hline\noalign{\smallskip}
$5$ & $4.08$e$-7$ \\
$6$ & $9.59$e$-9$ \\
$7$ & $1.69$e$-10$ \\
$10$ & $7.00$e$-14$ \\
$20$ & $6.07$e$-14$ \\
$30$ & $5.85$e$-14$ \\
\noalign{\smallskip}\hline
\end{tabular}
\quad
\begin{tabular}{cc}
\hline\noalign{\smallskip}
$s$ & ${err\_}\psi_{3,s}$ \\ 
\noalign{\smallskip}\hline\noalign{\smallskip}
$10$ & $1.34$e$-7$\\
$20$ & $1.27$e$-9$\\
$30$ & $7.92$e$-11$\\
$40$ & $1.09$e$-11$\\
$50$ & $2.32$e$-12$\\
$100$ & $1.55$e$-14$\\
\noalign{\smallskip}\hline
\end{tabular}
\end{center}
\end{table}

\begin{table}
\begin{center}
\caption{Errors for Example (\ref{test-tridiag})}\label{t1}
\begin{tabular}{cccc}
\hline\noalign{\smallskip}
$d$& ${err\_p}_{53}$ & ${err\_(7,7)}$-Pad\'e & ${err\_\psi}_{3, 50}$ \\ 
 \noalign{\smallskip}\hline\noalign{\smallskip}
$256$ & $2.29$e$-1$ & $2.92$e$-11$ & $7.54$e$-13$\\
$512$ & $2.29$e$-1$ & $2.92$e$-11$ & $7.54$e$-13$\\
$1024$ & $2.30$e$-1$& $2.92$e$-11$ & $7.54$e$-13$ \\
$2048$ & $2.30$e$-1$ & $2.92$e$-11$ & $7.54$e$-13$\\
\noalign{\smallskip}\hline
\end{tabular}
\end{center}
\end{table}

\begin{table}
\begin{center}
\caption{Errors for Example (\ref{test-invtridiag})}\label{t2}
\begin{tabular}{ccccc}
\hline\noalign{\smallskip}
$d$& ${err\_p}_{53}$& ${err\_(7,7)}$-Pad\'e & ${err\_p}_{3,50}$ & ${err\_p}_{3,50}$ with scaling \\ 
 \noalign{\smallskip}\hline\noalign{\smallskip}
$256$ & $1.60$e$-12$ & $1.60$e$-12$ & $1.60$e$-12$ & $1.60$e$-12$\\
$512$ & $5.55$e$-12$ & $2.55$e$-12$ & $2.55$e$-12$ &$2.55$e$-12$ \\
$1024$ & $2.14$e$-4$& $1.97$e$-11$ & $1.97$e$-11$ & $1.97$e$-11$ \\
$2048$ & $4.33$e$+173$ & $6.46$e$-11$  & $1.36$e$-2$ & $7.89$e$-10$\\
\noalign{\smallskip}\hline
\end{tabular}
\end{center}
\end{table}

\begin{table}
\begin{center}
\caption{Errors for Example (\ref{test-kms})}\label{t3}
\begin{tabular}{ c c c c }
\hline\noalign{\smallskip}
$d$& ${err\_p}_{53}$ & ${err\_(7,7)}$-Pad\'e & ${err\_p}_{3, 50}$  \\ 
\noalign{\smallskip}\hline\noalign{\smallskip}
$256$ & $5.37$e$-15$ & $3.06$e$-13$ & $2.38$e$-14$\\
$512$ & $6.21$e$-15$ & $3.14$e$-13$ & $2.43$e$-14$\\
$1024$ & $7.32$e$-15$& $3.16$e$-13$ & $2.45$e$-14$ \\
$2048$ & $1.15$e$-14$ & $3.17$e$-13$ & $2.46$e$-14$ \\
\noalign{\smallskip}\hline
\end{tabular}
\end{center}
\end{table}

\begin{table}
\begin{center}
\caption{Errors for Example (\ref{test-smoke})}\label{t4}
\begin{tabular}{c c c }
\hline\noalign{\smallskip}
${err\_p}_{53}$ & ${err\_(7,7)}$-Pad\'e & ${err\_\psi}_{3, 50}$ \\ 
 \noalign{\smallskip}\hline\noalign{\smallskip}
 $6.96$e$-16$ & $1.14$e$-13$ &$6.66$e$-16$\\
\noalign{\smallskip}\hline
\end{tabular}
\end{center}
\end{table}

In order to investigate further the behavior of  the different approximations  under the occurrence of possibly
complex eigenvalues we  have compared the accuracy of  polynomial and rational methods for approximating the matrix
$\psi_1(A)$ where $A=\gamma  F $ and  $F$ is the generator of the circulant matrix algebra, that is,
the companion matrix associated with the polynomial $z^d -1$.  Since we know that the eigenvalues of $F$ lie on the unit circle the
parameter $\gamma$ is used to estimate the convergence of the methods when the magnitude of eigenvalues increase.  Table \ref{t5}
illustrate the errors for the  case $d=1024$.  The divergence of the polynomial approximation
for $\gamma\geq 8$ is in accordance with the theoretical results.

\begin{table}
\begin{center}
\caption{Errors for  scaled companion matrices. Here $\phi_1(A)$ may be very ill-conditioned for large $\gamma$, so the Pad\'e approximation is not reported.}\label{t5}
\begin{tabular}{cccc}
\hline\noalign{\smallskip}
$\gamma$& ${err\_p}_{53}$  & ${err\_r}_{3, 50}$ & ${err\_r}_{3, 50}$ with scaling \\ 
\noalign{\smallskip}\hline\noalign{\smallskip}
$2$ & $7.72$e$-12$ & $7.72$e$-12$ & $7.72$e$-12$\\
$4$ & $7.52$e$-10$ & $7.52$e$-12$ & $7.52$e$-12$\\
$8$ & $8.44$e$+10$& $7.83$e$-12$ & $7.83$e$-12$\\
 $16$ & $1.58$e$+42$  & $4.46$e$-11$ & $7.54$e$-12$\\
 $32$ & $3.11$e$+73$ & $2.68$e$-9$ & $9.53$e$-12$\\
 $64$ & $2.046$e$+105$  & $5.86$e$-7$ & $9.41$e$-12$\\
\noalign{\smallskip}\hline
\end{tabular}
\end{center}
\end{table}

Finally, for  $\gamma=64$ we consider in Table \ref{t6} the errors generated by rational approximations of increasing order.
The table suggests  that rational approximations of higher orders are suited to  give accurate  results  independently
of the magnitude of the eigenvalues of $A$. 

\begin{table}
\begin{center}
\caption{Errors for  rational approximations of increasing orders.}\label{t6}
\begin{tabular}{ccccc}
\hline\noalign{\smallskip}
$\gamma$& ${err\_r}_{3, 50}$& ${err\_r}_{3, 100}$ &  ${err\_r}_{3, 200}$ & ${err\_r}_{3, 400}$ \\ 
 \noalign{\smallskip}\hline\noalign{\smallskip}
$64$ & $5.86$e$-7$ & $4.65$e$-9$ & $5.87$e$-11$ & $2.24$e$-11$\\
\noalign{\smallskip}\hline
\end{tabular}
\end{center}
\end{table}

\section{An application to a multi-degree of freedom system}\label{sec:multidof}

As an example of application of the $\psi_1$ function to a concrete physical problem, we consider  
 a multi-degree of freedom system suggested in \cite{Fung1997}, Section 9.2. It is a 
 mass-spring oscillating system of $N$ masses $m_1,\ldots,m_N$ as in the following diagram:

\bigskip

\begin{tikzpicture}
\node[circle,fill=gray,inner sep=2.5mm] (a) at (0,0) {};
\draw[decoration={aspect=0.3, segment length=3mm, amplitude=3mm,coil},decorate] (-2,0) -- (a); 
\fill [pattern = north east lines] (-2.5,-0.5) rectangle (-2,1);
\draw[thick] (-2,-0.5) -- (-2,1);
\node at (0,-0.7) {$m_1$};
\node[circle,fill=gray,inner sep=2.5mm] (b) at (2,0) {};
\draw[decoration={aspect=0.3, segment length=3mm, amplitude=3mm,coil},decorate] (a) -- (b); 
\node at (2,-0.7) {$m_2$};
\draw[decoration={aspect=0.3, segment length=3mm, amplitude=3mm,coil},decorate] (b) -- (4,0); 
\filldraw (4.3,0) circle (0.5pt);
\filldraw (4.5,0) circle (0.5pt);
\filldraw (4.7,0) circle (0.5pt);
\node[circle,fill=gray,inner sep=2.5mm] (c) at (7,0) {};
\draw[decoration={aspect=0.3, segment length=3mm, amplitude=3mm,coil},decorate] (5,0) -- (c); 
\node at (7,-0.7) {$m_N$};
\draw[decoration={aspect=0.3, segment length=3mm, amplitude=3mm,coil},decorate] (c) -- (9,0);
\fill [pattern = north east lines] (9,-0.5) rectangle (9.5,1);
\draw[thick] (9,-0.5) -- (9,1);
\end{tikzpicture}

\bigskip

The elastic constants of the $N+1$ springs are denoted as $k_1,\ldots,k_{N+1}$, whereas $\lambda_1,\ldots,\lambda_N$ are the friction coefficients associated with each mass.

 The system is modeled by the second-order ordinary differential equation 
\begin{equation}
M \frac{d^2}{dt^2}{\bf x}(t) + C \frac{d}{dt}{\bf x}(t) +K {\bf x}(t) = {\bf f},\qquad t\in [0,\tau],\label{eq:second_order_dof}
\end{equation}
where $M$, $C$, $K$ are the mass, damping and stiffness matrices, respectively. The vector ${\bf f}$ is the external force, assumed to be constant, and ${\bf x}(t)$ is the displacement vector, that is, $x_i(t)$ gives the position of the $i$-th mass with respect to a local reference system. Both ${\bf f}$ and ${\bf x}(t)$ are unknown. Such a setup could be useful, for instance, if we need to determine electric charges associated with $m_1,\ldots,m_N$: these can be obtained by applying a uniform electric field to the system and finding the constant force exerted on the masses. To this end, we can choose the initial position and velocity of the masses, that is, ${\bf x}(0)$ and $\frac{d}{dt}{\bf x}(0)$, and let the system evolve for a time $\tau$. Then the final position and velocity ${\bf x}(\tau)$ and $\frac{d}{dt}{\bf x}(\tau)$ are measured. Equipped with these data, we seek to determine ${\bf f}$.

In this experiment, the physical system is simulated numerically, with an arbitrary choice of ${\bf f}$, to determine an ``exact'' solution ${\bf x}(t)$ and thus the initial and final values of position and velocity of the masses. This procedure guarantees that these boundary conditions, although overdetermined, are compatible. The goal is to retrieve the value of the external force ${\bf f}$. 

Now, equation \eqref{eq:second_order_dof} can be rewritten as a first-order problem as follows:
\begin{equation}
\frac{d}{dt}{\bf y}(t)=A{\bf y}(t)+{\bf p},\label{eq:first_order_dof}
\end{equation}
where
$$
{\bf y(t)}=\left[\begin{array}{c}
{\bf x(t)}\\ \frac{d}{dt}{\bf x}(t)\end{array}\right],\qquad
A=\left[\begin{array}{cc}
0 & I\\ -M^{-1}K & -M^{-1}C \end{array}\right],\qquad
{\bf p}=\left[\begin{array}{c}
0\\ {\bf f}\end{array}\right],
$$
as the mass matrix $M$ can be safely assumed to be invertible.
Consider boundary conditions  
\begin{equation}
{\bf y(0)}=\left[\begin{array}{c}
{\bf x(0)}\\ \frac{d}{dt}{\bf x}(0)\end{array}\right],\qquad
{\bf y(\tau)}=\left[\begin{array}{c}
{\bf x(\tau)}\\ \frac{d}{dt}{\bf x}(\tau)\end{array}\right],\label{eq:bc}
\end{equation}
which are known from the simulation. 

We now have a first-order differential problem defined by equation \eqref{eq:first_order_dof} with boundary conditions \eqref{eq:bc}, and we seek to determine the vector ${\bf p}$. Note that the computed ${\bf p}$ is expected to be formed by a zero block followed by the sought value of ${\bf f}$. 

This problem happens to be of the same kind as the model problem \eqref{bvp}, \eqref{bvpc}. As mentioned in the Introduction, the vector {\bf p} can be computed explicitly via equation \eqref{pformula}. The main computational effort when applying \eqref{pformula} consists in computing the product of $\psi_1(\tau A)$ times a vector. Here this is done using our mixed polynomial-rational approximation \eqref{eq:approx} for $\psi_1(\tau A)$, which requires to solve several linear systems with coefficient matrices given by diagonal shifts of a scalar multiple of $A^2$. Understanding the structure of $A$ and $A^2$ may help solve such linear systems via fast methods, rather than applying a slower, general-purpose solver.

With the hypotheses outlined above, the mass and the damping matrices are diagonal: 
$$M={\rm diag}(m_1,m_2,\ldots,m_N),\qquad 
C={\rm diag}(\lambda_1,\lambda_2,\ldots,\lambda_N),
$$
whereas the matrix $K$ has a symmetric tridiagonal form:
$$
K=\left[\begin{array}{ccccc}
k_1+k_2 & -k_2 &\\
-k_2 & k_2+k_3 & -k_3 \\
& \ddots & \ddots & \ddots &\\
&&-k_{N-1} & k_{N-1}+k_N & -k_N\\
&&& -k_N & k_N+k_{N+1} 
\end{array}\right].
$$
Clearly the matrix $A$ inherits a sparse/banded structure, which allows for a computationally cheap application of the mixed approximation formula. Indeed, the matrix $A^2$ has a $2\times 2$ block structure with tridiagonal blocks, so we can employ the well-known formula
\begin{equation}\label{inv2}
\left[\begin{array}{cc} 
\mathcal{A} & \mathcal{B}\\ \mathcal{C} & \mathcal{D}
\end{array}\right]^{-1}=
\left[\begin{array}{cc}
\mathcal{A}^{-1}+\mathcal{A}^{-1}\mathcal{B}\mathcal{S}^{-1}\mathcal{C}\mathcal{A}^{-1}&-\mathcal{A}^{-1}\mathcal{B}\mathcal{S}^{-1}\\
-\mathcal{S}^{-1}\mathcal{C}\mathcal{A}^{-1}&-\mathcal{S}^{-1}
\end{array}\right]
\end{equation}
and perform inversions using quasiseparable structure.\footnote{A detailed presentation of quasiseparable matrix structure is beyond the scope of this paper; we refer the interested reader to the book \cite{EGH_book}. For the purpose of this example let us recall that quasiseparability is a kind of matrix rank structure that allows for inversion in $O(N)$ operations, and that banded matrices belong to the quasiseparable class.} In particular, note that the Schur complement $\mathcal{S}=\mathcal{D}-\mathcal{C}\mathcal{A}^{-1}\mathcal{B}$ has quasiseparable rank at most three. Therefore, the computational cost of this approach is $O(N)$. While in general inversion methods based on the Schur complement may suffer from stability issues, such issues are not observed in this specific example.

\begin{remark}
Of course the structure of the matrix $A$ can be parameterized in different ways. For instance, one may observe that $A$ is banded and therefore also quasiseparable; note, however, that the bandwidth and the quasiseparability order increase with $N$. On the other hand, the quasiseparable generators are sparse themselves, which in practice may lead to computational savings.

In addition, if the masses in the physical system are all equal and have the same friction coefficient, and the springs all have the same elastic constant, then $A$ exhibits a low-order Toeplitz-like structure, which allows for fast inversion of $A^2/2\pi+k^2I$. 
\end{remark}

In this numerical test we take $\tau=1$ and $m_i=1$, $k_i=0.3$, $\lambda_i=0.1$ for all indices $i$. As initial conditions for $t=0$ we choose a displacement of $0.5$ for all masses and zero velocity. The constant external force is set to $0.5$, that is, ${\bf f}=[0.5,\ldots,0.5]^T$. In order to simulate the evolution of the physical model, we need to integrate the differential equation \eqref{eq:first_order_dof}: a preliminary computation via the MATLAB command {\tt ode45}, with absolute and relative tolerances set at $1e-13$, yields the boundary condition at $t=\tau$. With this setup we compute the vector ${\bf p}$ via mixed approximation and compare it to the ``exact'' one. The methods examined in this example are:
\begin{itemize}
\item
unstructured mixed polynomial-rational approximation (that is, computation of $\psi_{3,10}$ in \eqref{eq:approx} without taking advantage of the structure of $A$),
\item
structured block mixed polynomial-rational approximation (that is, computation of $\psi_{3,10}$ in \eqref{eq:approx}, where matrix inversion is performed via \eqref{inv2} in combination with quasiseparable inversion algorithms),
\item
approximation of $\psi_1(A)$ via EXPINT,
\item
approximation of $\psi_1(A)$ via {\tt expm}.
\end{itemize}
The results are shown in Tables \ref{table:N50}--\ref{table:N1000}, for $N=50, 100, 500, 1000$. The tables report the norm of the $(1:N)$-block of the computed vector ${\bf p}$, which should ideally be zero, and the absolute and relative errors on the $(N+1:2N)$-block corresponding to the force vector. The quality of the {\tt expm} approximation tends to deteriorate for growing $N$, whereas both the structured and the unstructured mixed approximation are as accurate as EXPINT. 

\begin{table}
\begin{center}
\caption{Errors for $N=50$.}\label{table:N50}
\begin{tabular}{lccc}  
\hline\noalign{\smallskip}
& $\|{\bf p}(1:N)\|_2$ & abs. err. on ${\bf f}$ & rel. err. on ${\bf f}$\\
\noalign{\smallskip}\hline\noalign{\smallskip}
Unstructured mixed & $1.52$e$-14$ & $1.06$e$-14$ &  $2.98$e$-15$\\
$2\times 2$ QS mixed & $1.51$e$-14$ & $1.05$e$-14$ & $2.98$e$-15$\\
EXPINT & $1.51$e$-14$ & $1.06$e$-14$ &  $2.99$e$-15$\\
{\tt expm} & $1.77$e$-13$ & $7.45$e$-13$ &  $2.11$e$-13$\\
\noalign{\smallskip}\hline
\end{tabular}
\end{center}
\end{table}

\begin{table}
\begin{center}
\caption{Errors for $N=100$.}\label{table:N100}
\begin{tabular}{lccc}  
\hline\noalign{\smallskip}
& $\|{\bf p}(1:N)\|_2$ & abs. err. on ${\bf f}$ & rel. err. on ${\bf f}$\\
\noalign{\smallskip}\hline\noalign{\smallskip}
Unstructured mixed & $1.52$e$-14$ & $1.06$e$-14$ &  $2.13$e$-15$\\
$2\times 2$ QS mixed & $1.51$e$-14$ & $1.05$e$-14$ & $2.11$e$-15$\\
EXPINT & $1.52$e$-14$ & $1.09$e$-14$ &  $2.17$e$-15$\\
{\tt expm} & $9.69$e$-13$ & $3.88$e$-12$ &  $7.76$e$-13$\\
\noalign{\smallskip}\hline
\end{tabular}
\end{center}
\end{table}

\begin{table}
\begin{center}
\caption{Errors for $N=500$.}\label{table:N500}
\begin{tabular}{lccc}  
\hline\noalign{\smallskip}
& $\|{\bf p}(1:N)\|_2$ & abs. err. on ${\bf f}$ & rel. err. on ${\bf f}$\\
\noalign{\smallskip}\hline\noalign{\smallskip}
Unstructured mixed & $1.52$e$-14$ & $1.15$e$-14$ &  $1.03$e$-15$\\
$2\times 2$ QS mixed & $1.51$e$-14$ & $1.06$e$-14$ & $9.48$e$-16$\\
EXPINT & $1.54$e$-14$ & $1.31$e$-14$ &  $1.18$e$-15$\\
{\tt expm} & $5.05$e$-11$ & $1.74$e$-10$ &  $1.56$e$-11$\\
\noalign{\smallskip}\hline
\end{tabular}
\end{center}
\end{table}

\begin{table}
\begin{center}
\caption{Errors for $N=1000$.}\label{table:N1000}
\begin{tabular}{lccc}  
\hline\noalign{\smallskip}
& $\|{\bf p}(1:N)\|_2$ & abs. err. on ${\bf f}$ & rel. err. on ${\bf f}$\\
\noalign{\smallskip}\hline\noalign{\smallskip}
Unstructured mixed & $1.52$e$-14$ & $1.26$e$-14$ &  $7.96$e$-16$\\
$2\times 2$ QS mixed & $1.51$e$-14$ & $1.07$e$-14$ & $6.75$e$-16$\\
EXPINT & $1.57$e$-14$ & $1.47$e$-14$ &  $9.28$e$-16$\\
{\tt expm} & $3.25$e$-10$ & $9.96$e$-10$ &  $6.30$e$-11$\\
\noalign{\smallskip}\hline
\end{tabular}
\end{center}
\end{table}

\section{Bounds on the decay of the reciprocal of the $\phi_1$-function}\label{three}

In this section we investigate the approximate rank structure of  $\psi_1(A)$  for a suitable $A$. Specifically,
as an application of Theorem \ref{main} we can deduce {\em a priori} bounds on the decay of the $\psi_1$-function applied to  symmetric banded matrices.

Now, let $A\in\mathbb{R}^{d\times d}$ be a symmetric banded matrix. Denote as $m$ the half-bandwidth of $A$,
that is, $A_{i,j}=0$ if $|i-j|>m$. It is well-known that the off-diagonal entries of $\psi_1(A)$
exhibit a decay behavior in absolute value (the same is true of any other function of $A$
that is well-defined and sufficiently regular \cite{BG1}). We can use the  $(n,s)$-mixed polynomial-rational
approximation \eqref{eq:approx}
to give bounds on this decay behavior.

Define 
\[
r_{n,s}(z)=2(-1)^n \left(\frac{z}{2\pi}\right)^{2(n+1)}\sum_{k=1}^s \frac{1}{k^{2n}\left(\left(\frac{z}{2\pi}\right)^{2}+k^2\right)},
\]
\[
p_n(z)=1-\frac12 z+\sum_{i=0}^{n-1}z^{2(i+1)}\frac{B_{2(i+1)}}{(2(i+1))!}
\]
which we will call the rational and the polynomial  part of \eqref{eq:approx}, respectively,  and let 
$$
\varepsilon_{n,s}(z)=\psi_1(z)-p_n(z)-r_{n,s}(z)
$$
 be the $(n,s)$-approximation error. We have
\begin{equation}\label{eq:bound1}
|[\psi_1(A)]_{i,j}|\leq |[p_n(A)]_{i,j}|+|[r_{n,s}(A)]_{i,j}|+|[\varepsilon_{n,s}(A)]_{i,j}|.
\end{equation}
Observe that $p_n(A)$ is a banded matrix with half-bandwidth $2nm$. So, if we choose $i,j$ such that $|i-j|>2nm$, then
$|[p_n(A)]_{i,j}|=0$ and we only need to focus on the rational and error terms.

For the rational term, let us start by giving a bound on
\[
\tilde{r}_{n,s}(A):=\sum_{k=1}^s \frac{1}{k^{2n}}\left(\left(\frac{A}{2\pi}\right)^{2}+k^2I_d\right)^{-1}.
\]
The matrix $A_k=\left(\frac{A}{2\pi}\right)^{2}+k^2I_d$ is positive definite with
semi-bandwidth $2m$, and several exponential decay bounds for the inverse of a positive definite
matrix have been proposed in the literature.
Prop. 2.2 from \cite{DMS}, for instance, gives
\[
|[(A_k)^{-1}]_{i,j}|\leq C_k \lambda_k^{|i-j|},
\]
where
\begin{equation}\label{bb1}
a_k=k^2, \quad b_k=\left(\frac{\rho(A)}{2\pi}\right)^2 +k^2,  \quad r_k=\frac{b_k}{a_k},
\end{equation}
\begin{equation}\label{bb2}
\lambda_k=\left(\frac{\sqrt{r_k}-1}{\sqrt{r_k}+1}\right)^{1/2m}, \quad C_k=\max\left\{a_k^{-1},
\frac{(1+\sqrt{r_k})^2}{2a_kr_k}\right\}, 
\end{equation}
where $\rho(A)$ is the spectral radius of $A$ and $0<a_k<b_k$ are such that the spectrum of $A_k$ is contained in $[a_k,b_k]$.
Therefore we have
\[
\left|\left[\sum_{k=1}^s \frac{1}{k^{2n}}\left(\left(\frac{A}{2\pi}\right)^{2}+k^2I_d\right)^{-1}\right]_{i,j}\right|\leq
\sum_{k=1}^s\frac{1}{k^{2n}}C_k \lambda_k^{|i-j|}
\]
for all indices $i,j$.
Now recall that $\left(\frac{A}{2\pi}\right)^{2(n+1)}$ is a banded matrix of bandwidth $2m(n+1)$. So we have:
\begin{eqnarray*}
\left|[r_{n,s}(A)]_{i,j}\right|=\left|\sum_{\nu=1}^d[\tilde{r}_{n,s}(A)]_{i,\nu}\left[\left(\frac{A}{2\pi}\right)^{2(n+1)}\right]_{\nu,j}\right|\nonumber\\
=\left|\sum_{\nu=j-2m(n+1)}^{j+2m(n+1)}[\tilde{r}_{n,s}(A)]_{i,\nu}\left[\left(\frac{A}{2\pi}\right)^{2(n+1)}\right]_{\nu,j}\right|\nonumber\\
\leq\sum_{\nu=j-2m(n+1)}^{j+2m(n+1)}\left\|\frac{A}{2\pi}\right\|_2^{2(n+1)}\left(\sum_{k=1}^s\frac{C_k}{k^{2n}}\lambda_k^{|i-\nu|}\right)\nonumber\\
=\left\|\frac{A}{2\pi}\right\|_2^{2(n+1)}\sum_{\nu=j-2m(n+1)}^{j+2m(n+1)}\sum_{k=1}^s\frac{C_k}{k^{2n}}\lambda_k^{|i-\nu|},\label{bound_rational}
\end{eqnarray*}
where in the sums over $\nu$ it is understood that $1\leq\nu\leq d$.

Let us now bound the error term. Define
\[
\varepsilon_{n,s}(A)=\tilde{\varepsilon}_{n,s}(A)\left(\frac{A}{2\pi}\right)^{2(n+1)}\quad\textnormal{where}\;\;
\tilde{\varepsilon}_{n,s}(A)=\sum_{k=s+1}^\infty \frac{1}{k^{2n}}\left(\left(\frac{A}{2\pi}\right)^{2}+k^2I_d\right)^{-1}.
\]
Therefore we find that 
\[
|[{\varepsilon}_{n,s}(A)]_{i,j}|\leq\|{\varepsilon}_{n,s}(A)\|_2\leq\|\tilde{\varepsilon}_{n,s}(A)\|_2
\left\|\frac{A}{2\pi}\right\|_2^{2(n+1)}.
\]
Let us bound $\|\tilde{\varepsilon}_{n,s}(A)\|_2$.
Let $A=UDU^H$ be the eigendecomposition of $A$ and denote the spectrum of $A$ as $\sigma(A)$; recall that $\sigma(A)$ is real.
We have 
\[\|\tilde{\varepsilon}_{n,s}(A)\|_2\leq\|\tilde{\varepsilon}_{n,s}(D)\|_2=\max_{x\in\sigma(A)}|\tilde{\varepsilon}_{n,s}(x)|
\]
 and moreover
\[
|\tilde{\varepsilon}_{n,s}(x)|=\left|\sum_{k=s+1}^\infty\frac{1}{k^{2n}}\left(\left(\frac{x}{2\pi}\right)^2+k^2\right)^{-1}\right|\leq
\sum_{k=s+1}^\infty\frac{1}{k^{2n+2}}=\zeta(2n+2)-\sum_{k=1}^s\frac{1}{k^{2n+2}},
\]
from which we deduce
\[
  |[{\varepsilon}_{n,s}(A)]_{i,j}|\leq\left\|\frac{A}{2\pi}\right\|_2^{2(n+1)}\left(\zeta(2n+2)-\sum_{k=1}^s\frac{1}{k^{2n+2}}
  \right),\label{bound_error}
\]
where $\zeta(s)$ is the Riemann zeta function.

Summing up  the following estimates are obtained  for the entries of $\psi_1(A)$.
\begin{theorem}\label{btheo}
  Let $A\in \mathbb R^{d\times d}$ be  a symmetric banded matrix with half-bandwidth $m$.  For all
  $(n, s)\in \mathbb N\times \mathbb N$  it holds 
\[
 |[\psi_1(A)]_{i,j}|\leq |[r_{n,s}(A)]_{i,j}|+|[\varepsilon_{n,s}(A)]_{i,j}|,\quad |i-j|>2mn,
\]
where
\[
|[r_{n,s}(A)]_{i,j}|\leq \left\|\frac{A}{2\pi}\right\|_2^{2(n+1)}\sum_{\nu=j-2m(n+1)}^{j+2m(n+1)}\sum_{k=1}^s\frac{C_k}{k^{2n}}\lambda_k^{|i-\nu|}
\]
and
\[
|[\varepsilon_{n,s}(A)]_{i,j}|\leq \left\|\frac{A}{2\pi}\right\|_2^{2(n+1)}\left(\zeta(2n+2)-\sum_{k=1}^s\frac{1}{k^{2n+2}}\right),
\]
and $C_k$ and $\lambda_k$ are given in \eqref{bb1},\eqref{bb2}.
\end{theorem}

To illustrate the significance of these bounds  we  present  in Figure \ref{fig:comparison}
numerical comparisons  with  other
existing bounds  deduced from \cite{BG1}.  Recall that these latter estimates 
are based on a  theoretical result on the best
degree-$k$ polynomial approximation of the function $\psi_1(z)$  on $[-1,1]$
that cannot be  explicitly computed. The corresponding
best polynomial approximation error  satisfies 
\[
E_k(\psi_1)\leq\frac{2M(\chi)}{\chi ^k(\chi-1)}
\]
and depends on a parameter $\chi$ that defines a Bernstein ellipse in the complex plane, where the function is analytic.
For the case considered in Figure \ref{fig:comparison} a good choice is $\chi=12$. 
If the spectrum of the matrix is not contained in $[-1,1]$,
one needs to scale the matrix, that is, apply the function $\psi_{1,\xi}(z)=\frac{\xi z}{e^{\xi z}-1}$ to $A/\xi$,
for a suitable choice of $\xi$. Then the poles of the function closest to zero are at $\pm \frac{2\pi i}{\xi}$;
the minor semi axis $\beta$ of the ellipse should be chosen slightly smaller than $\frac{2\pi}{\xi}$ and
$\chi=\beta+\sqrt{\beta^2+1}$. 

\begin{figure}
  \begin{center}
\includegraphics[width=0.9\textwidth]{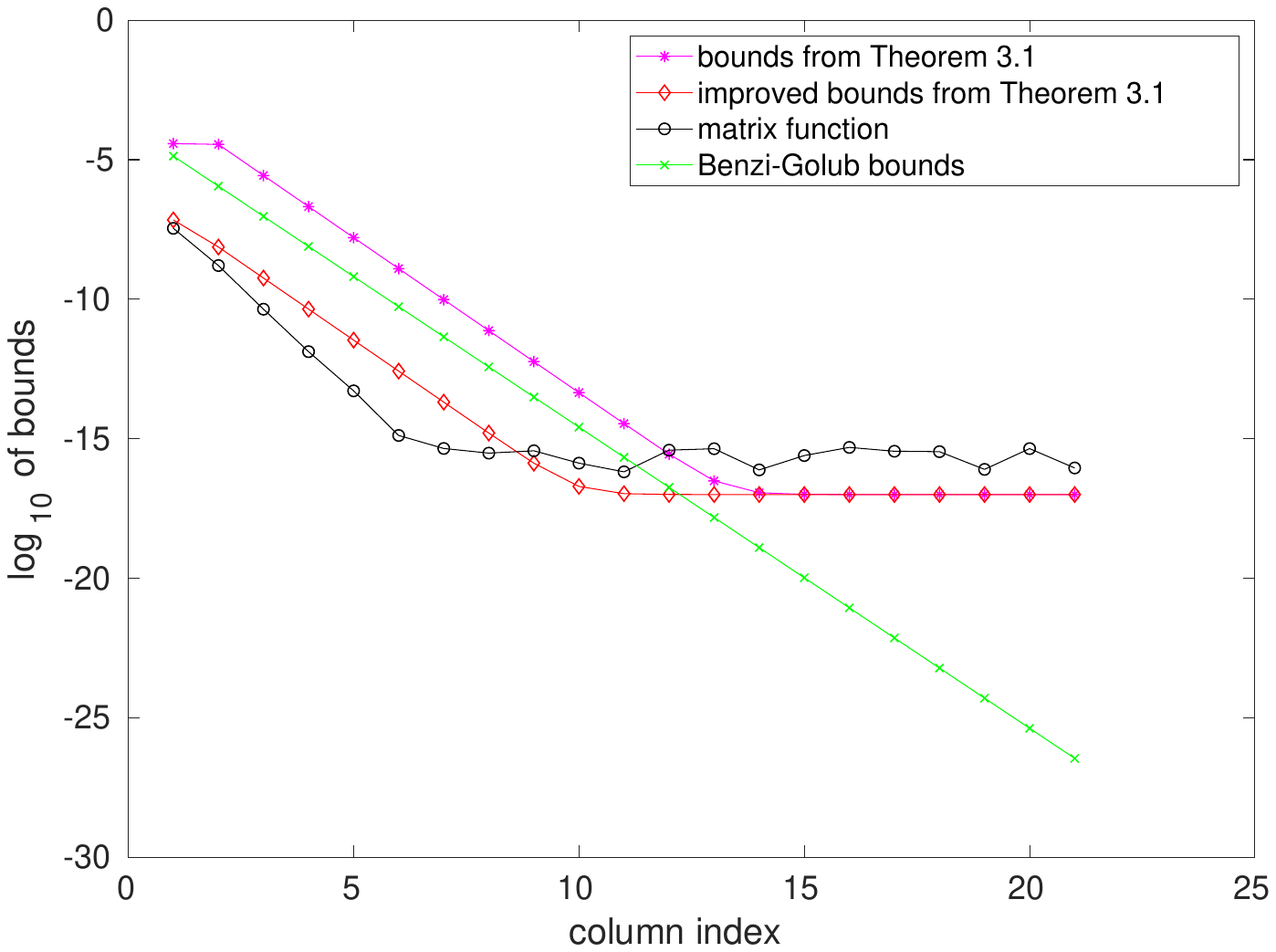}
\caption{Bounds for a symmetric tridiagonal Toeplitz matrix $A$ with spectrum lying in  $[-1,1]$. The figure shows the behavior of bounds and of the absolute value of the actual matrix function on the first matrix row. In particular, bounds from Theorem 3.1 are shown in purple, whereas the bounds in red are obtained by explicitly computing the product between $\left|(\frac{A}{2\pi})^{2(n+1)}\right|$ and the matrix of bounds on $r_{n,s}(A)$.}\label{fig:comparison}
\end{center}
  \end{figure}

We see that  the  proposed mixed polynomial-rational
approximation  and the best polynomial approximation exhibit a similar decaying profile.

\section{Conclusion and Future Work}\label{four}
In this paper we have introduced a family of rational approximations of the  
reciprocal of the $\phi_1$-function encountered in exponential integration  
methods. This family extends customary approximations based on  the Taylor 
series by showing better convergence properties. Therefore, the novel formulas 
are particularly suited when applied for computing the reciprocal of the 
$\phi_1$  matrix function  of a structured matrix  admitting fast and 
numerically robust linear solvers. Mixed polynomial-rational approximations of 
a meromorphic function  based on  the Dunford-Cauchy integral formula that are 
suited for computation with rank-structured matrices have been recently 
proposed in \cite{MaRo}. Theoretical and  computational comparisons between  
the two families of approximations of $\psi_1(z)={\phi_1(z)}^{-1}$ is an 
ongoing work. Also a more detailed comparison of the approaches  based on the  
Mittag-Leffler theorem and the rational Carath\'eodory-Fej\'er approximation 
\cite{ST} for evaluating $\psi_1(z)$ would be interesting. 

Another natural continuation of our results in Section 2 is the complete
numerical solution of the inverse and nonlocal problems for differential 
equations as (\ref{bvp}), (\ref{bvpc}) and (\ref{laris}).



\begin{acknowledgements} 
Part of the first author's work was done while at XLIM--MATHIS, Universit\'e de
Limoges (UMR CNRS 7252) and on secondment in the AriC group at LIP, ENS de Lyon
(CNRS, ENS Lyon, Inria, UCBL).

We thank I.~V.~Tikhonov for his valuable remarks.
\end{acknowledgements}

\bibliographystyle{spmpsci}
\bibliography{matrixbib}

\end{document}